\documentclass[11pt]{amsart}
\usepackage{amssymb}
\usepackage{bm}
\usepackage{graphicx}
\usepackage[centertags]{amsmath}
\usepackage{amsfonts}
\usepackage{amsthm}
\usepackage{color}

\theoremstyle{plain}
\theoremstyle{definition}
\newtheorem{theorem}{Theorem}[section]
\newtheorem{lemma}[theorem]{Lemma}
\newtheorem{proposition}[theorem]{Proposition}
\newtheorem{corollary}[theorem]{Corollary}
\newtheorem{definition}[theorem]{Definition}
\newtheorem{remark}[theorem]{Remark}

\newtheorem{notation}[theorem]{Notation}

\newcommand{\e}{\varepsilon}

\DeclareMathOperator{\N}{\mathbb{N}}

\DeclareMathOperator{\Q}{\mathbb{Q}}
\DeclareMathOperator{\R}{\mathbb{R}}

\DeclareMathOperator{\supp}{supp}

\DeclareMathOperator{\ran}{ran} \DeclareMathOperator{\rank}{rank}

\DeclareMathOperator{\dist}{dist}

\newcommand{\bp}{\begin{proof}}
\newcommand{\ep}{\end{proof}}
\newcommand{\bd}{\begin{definition}}
\newcommand{\ed}{\end{definition}}
\newcommand{\bt}{\begin{theorem}}
\newcommand{\et}{\end{theorem}}
\newcommand{\bpr}{\begin{proposition}}
\newcommand{\epr}{\end{proposition}}
\newcommand{\bc}{\begin{corollary}}
\newcommand{\ec}{\end{corollary}}
\newcommand{\bl}{\begin{lemma}}
\newcommand{\el}{\end{lemma}}
\newcommand{\br}{\begin{remark}}
\newcommand{\er}{\end{remark}}
\newcommand{\be}{\begin{enumerate}}
\newcommand{\ee}{\end{enumerate}}
\long\def\symbolfootnote[#1]#2{\begingroup%
\def\thefootnote{\fnsymbol{footnote}}\footnote[#1]{#2}\endgroup}

\begin{document}
\title{Bourgain-Delbaen $\mathcal{L}^{\infty}$-sums of Banach spaces}
\author{Despoina Zisimopoulou}
\footnote{This research has been co-financed by the European Union (European Social Fund – ESF) and Greek national funds through the Operational Program "Education and Lifelong Learning" of the National Strategic Reference Framework (NSRF) - Research Funding Program: Heracleitus II. Investing in knowledge society through the European Social Fund.}

\address{National Technical University of Athens, Faculty of Applied Sciences,
Department of Mathematics, Zografou Campus, 157 80, Athens,
Greece} \email{dzisimopoulou@hotmail.com}
\maketitle

\symbolfootnote[0]{\textit{2010 Mathematics Subject
Classification:} Primary 46B03, 46B25, 46B28}
\symbolfootnote[0]{\textit{Key words:} Bourgain Delbaen method, $\mathcal{L}_{\infty}$ spaces, strictly quasi prime spaces, horizontally compact operators}

\begin{abstract} Motivated by a problem stated by S.A.Argyros and Th. Raikoftsalis, we introduce a new class of Banach spaces.
Namely, for a sequence of separable Banach spaces
$(X_n,\|\cdot\|_n)_{n\in\N}$, we define the Bourgain Delbaen
$\mathcal{L}^{\infty}$-sum of the sequence
$(X_n,\|\cdot\|_n)_{n\in\N}$ which is a Banach space $\mathcal{Z}$
constructed with the Bourgain-Delbaen method. In
particular,
 for every $1\leq p<\infty$, taking $X_n=\ell_p$ for every $n\in\N$ the aforementioned space $\mathcal{Z}_p$
  is strictly quasi prime and admits $\ell_p$ as a complemented subspace. We study the operators acting on $\mathcal{Z}_p$ and we prove
  that for every $n\in\N$, the space $\mathcal{Z}^n_p=\sum_{i=1}^n\oplus \mathcal{Z}_p$ admits exactly $n+1$, pairwise not isomorphic, complemented subspaces.
\end{abstract}
\section{Introduction}

There has been an extensive study of Schauder sums of sequences of
Banach spaces $(X_n,\|\cdot\|_n)$ (\cite{AF},\cite{AR}) with many interesting applications, depending on the choices of the spaces $X_n$ and the external norm. In particular, in \cite{AF} the authors defined Schauder sums of arbitrary sequence of separable Banach spaces $(\sum_{n\in\N}\oplus X_n)_{GM}$ where the external norm is based on the Gowers Maurey norm \cite{GM1}.

In \cite{AR} the authors studied further the spaces $\mathfrak{X}_p=(\sum_{n\in\N}\oplus \ell_p)_{GM}$ for $1\leq p<\infty$, $\mathfrak{X}_0=(\sum_{n\in\N}\oplus c_0)_{GM}$ as well as the space of their bounded, linear operators. Moreover, in the same paper it was proved that for $\mathfrak{X}=\mathfrak{X}_p$ or $\mathfrak{X}_0$ the space $\mathfrak{X}^n=(\sum_{i=1}^n\oplus\mathfrak{X})_{\infty}$ admits at least $n+1$, up to isomorphism, complemented subspaces and it was stated as an open problem whether they are exactly $n+1$. We do not give an affirmative answer to this problem but instead, following the basic scheme of the authors in \cite{AH} we present a method of constructing for every $n\in\N$ Banach spaces with exactly $n+1$, up to isomorphism, complemented subspaces. In particular, we shall define and construct Schauder sums of sequences of Banach spaces with an external norm that is based on the original Bourgain-Delbaen norm (\cite{BD}).

We now give a description of how this paper is organized. In section 2, given a sequence of separable Banach spaces $(X_n,\|\cdot\|_n)_{n\in\N}$, we define the
Bourgain Delbaen (BD) -$\mathcal{L}^{\infty}$-sum of $(X_n,\|\cdot\|_n)_{n\in\N}$, denoted as $\mathcal{Z}=(\sum_{n\in\N}\oplus X_n)_{BD}$. This space is defined along with a sequence of pairwise disjoint and finite subsets of $\N$, the so called "Bourgain-Delbaen" sets $(\Delta_n)_{n\in\N}$. $\mathcal{Z}$ has a Schauder Decomposition $(Z_n)_{n\in\N}$ and there exists a constant $C>0$ such that $Z_n\simeq^C (X_n\oplus\ell_{\infty}(\Delta_n))_{\infty}$ for every $n\in\N$.

In section 3 we study $\mathcal{Z}^*$ and we show that if the Schauder Decomposition $(Z_n)_{n\in\N}$ of $\mathcal{Z}$ is shrinking then $\mathcal{Z}^*$ can be identified with $(\sum_{n\in\N}\oplus(X_n^*\oplus\ell_1(\Delta_n))_1)_1$. In Section 4 we describe in detail the construction of $\mathcal{Z}$. We also study the special case where the Bourgain Delbaen external norm is the Argyros-Haydon norm in
\cite{AH}, yielding spaces $\mathcal{Z}$ which we denote by $(\sum_{n\in\N}\oplus X_n)_{AH}$. We devote Sections 5,6 and 7 into proving the following.
 \begin{theorem}\label{1}
Let $(X_n,\|\cdot\|_n)_n$ be a sequence of separable Banach spaces and $\mathcal{Z}=(\sum_{n=1}^{\infty}\oplus X_n)_{AH}$. Then the following hold:
\begin{enumerate}
\item The space $\mathcal{Z}$ admits a shrinking Schauder Decomposition.
\item Every block (with respect to $(Z_k)_{k\in\N}$ ) sequence $(x_n)_{n\in\N}$ generates an HI subspace,
i.e. the subspace $\overline{<x_n:\ n\in\N>}$ of $Z$ is HI.
\item Assume that for every $n\in\N$, either $\ell_1$ does not embed in $X_n^*$ or $X_n$ has the Schur property.  Then, for every bounded, linear operator $T$ on $\mathcal{Z}$ there exists
a scalar $\lambda\in\R$ such that the operator $T-\lambda I$ on $\mathcal{Z}$ is
horizontally compact, i.e. for every bounded, block (with respect to $(Z_n)_{n\in\N}$ )
sequence $(z_k)_{k\in\N}$ in $Z$, $\|(T-\lambda I)z_k\|\to 0$.
\end{enumerate}
\end{theorem}

In Section 8 we prove that $\mathcal{Z}_p=(\sum_{n\in\N}\oplus \ell_p)_{AH}$ for $1\leq p<\infty$ is strictly quasi prime and contains isomorphically
$\ell_p$ as a complemented subspace. We recall (see \cite{AH}) that a Banach space $X$ is strictly quasi prime if there exists a subspace $Y$ of $X$ not isomorphic to $X$ such that $X$ admits a unique non trivial decomposition as $Y\oplus X$.

Our main result is proved in Section 9:

\begin{theorem}\label{2}
For every $1\leq p<\infty$ and $n\in\N$ the space $\mathcal{Z}_p^n=(\sum_{i=1}^n\oplus
\mathcal{Z}_p)_{\infty}$ has exactly $n+1$ up to isomorphism complemented subspaces.
\end{theorem}

We must mention that similar results were obtained by many authors using different techniques. Namely, P. Wojtasczszyk in \cite{W} and P. Wojtasczszyk, I.S. Edelstein in \cite{EW} proved that for every $n\in\N$ there exists a separable Banach space with exactly $2^n-1$ complemented subspaces. Moreover, as it is pointed out in \cite{FG}, W.T. Gowers and B. Maurey in \cite{GM} constructed for every $n\in\N$ a Banach space $X_n$ that has exactly $n$ up to isomorphism complemented subspaces. Our approach is more direct using a combination of techniques which are presented in \cite{AH} and \cite{AR}.

This introduction cannot end without giving my special thanks to my advisor and Professor S. A. Argyros for his helpful advice as well as to I. Gasparis for useful conversations regarding the results of this work.

\section{The definition of a BD-$\mathcal{L}^{\infty}$-sum of Banach spaces}
We start by giving the needed terminology.
\begin{notation}\label{terminology}
Let $(E_n,\|\cdot\|_{E_n})_{n=1}^{\infty}$ be sequences of separable
Banach spaces. For $I$ interval of $\N$ or $I=\N$ we consider direct sums $(\sum_{n\in I}\oplus E_n)_{\infty}$, we use vectors as $\overrightarrow{x},\overrightarrow{y},\overrightarrow{z}$ to represent their elements. For $\overrightarrow{x}\in(\sum_{n\in I}\oplus E_n)_{\infty}$ we denote by $\overrightarrow{x}(n)$ the $n$-th coordinate of $\overrightarrow{x}$ in $E_n$ and the norm is defined as $\|\overrightarrow{x}\|_{\infty}=\sup_{n\in I}\|\overrightarrow{x}(n)\|_{E_n}$. In a similar manner we consider dual direct sums $(\sum_{n\in I}E_n^*)_1$ consisting of elements functions vectors which we denote by $\overrightarrow{f},\overrightarrow{g},\overrightarrow{h}$, etc. For an element $\overrightarrow{f}$ we define $\|\overrightarrow{f}\|_1=\sum_{n\in I}\|\overrightarrow{f}(n)\|_{E_n^*}$ and for $\overrightarrow{x}\in(\sum_{k\in I}\oplus E_k)_{\infty}$ we denote by $\overrightarrow{f}(\overrightarrow{x})$ the inner product $\sum_{n\in I}\overrightarrow{f}_n(\overrightarrow{x}_n)$, where $\overrightarrow{f}_n=\overrightarrow{f}(n)$ and $\overrightarrow{x}_n=\overrightarrow{x}(n)$.

For every $I$ finite interval of $\N$ we denote by $R_{I}$ the natural surjections $R_{I}:(\sum_{n\in\N}\oplus E_n)_{\infty}\to(\sum_{n\in I}\oplus E_n)_{\infty}$ defined as $R_{I}(\overrightarrow{z})=(\overrightarrow{z}(n))_{n\in I}$. We use a "star" notation $R_I^*$ to regard in a similar manner the natural surjection on the duals, i.e. $R_{I}^*:(\sum_{n\in\N}\oplus E_n^*)_1\to(\sum_{n\in I}\oplus E_n^*)_1$

For $I, J$ intervals of $\N$ we say that $I, J$ are successive denoted as $I<J$ if $\max I\leq\min J$. Let $I_1<I_2<I_3$ such that $\max I_i+1=\min I_{i+1}$ for every $i=1,2$.
For vectors $\overrightarrow{x}_1,\overrightarrow{x}_2,\overrightarrow{x}_3$ such that $\overrightarrow{x}_i\in(\sum_{k\in I_i}\oplus E_n)_{\infty}$  we denote by $(\overrightarrow{x}_1,\overrightarrow{x}_2,\overrightarrow{x}_3)$ the vector $\overrightarrow{y}\in(\sum_{n=\min I_1}^{\max I_3}\oplus E_n)_{\infty}$  defined as $\overrightarrow{y}(n)=\overrightarrow{x_i}(n)$ whenever $n\in I_i$. Similarly we define vectors $(\overrightarrow{f}_1,\overrightarrow{f}_2,\overrightarrow{f}_3)$ in $(\sum_n\oplus E_n^*)_1$ where $\overrightarrow{f}_i\in(\sum_{k\in I_i}\oplus E_n^*)_1$.
\end{notation}

We now recall the definition of classical $\mathcal{L}^{\infty}$ spaces (see \cite{AH}),a generalised form of which will be used in order to define the new class of spaces, namely the BD $\mathcal{L}_{\infty}$ sums of Banach spaces.
\begin{definition}
We say that a separable Banach space $X$ is $\mathcal{L}_{\infty,C}$ where $C>0$ is a constant, if there exists a strictly increasing sequence $(Y_n)_{n\in\N}$ of subspaces of $X$ such that $Y_n$ is $C$- isomorphic to $\ell_{\infty}(\dim Y_n)$ for every $n\in\N$ and $X=\overline{\cup_{n\in\N}Y_n}$.
\end{definition}
We also recall that for Banach spaces $Z,W$ and a constant $M>0$ we say that $Z$ is $M$- isomorphic to $W$ if there exists $T:Z\to W$ such that $\|T\|\|T^{-1}\|\leq M$.

\begin{definition}\label{mathcal}
Let $(X_n,\|\cdot\|_n)_{n\in\N}$ be a sequence of separable Banach spaces. A Banach space $\mathcal{Z}$ is called a
Bourgain Delbaen(BD)-$\mathcal{L}^{\infty}$-sum of the sequence
$(X_n,\|\cdot\|_n)_n$, denoted as $\mathcal{Z}=(\sum_{n=1}^{\infty}\oplus X_n)_{BD}$, if there exists a sequence
$(\Delta_n)_{n\in\N}$ of finite, pairwise disjoint subsets of $\N$
and the following hold:
\begin{enumerate}
\item The space $\mathcal{Z}$ is a subspace of
$(\sum_{n=1}^{\infty}\oplus(X_n\oplus\ell_{\infty}(\Delta_n))_{\infty})_{\infty}$. \item There exist $C>0$ and for every $n$,
a linear operator
\[i_n:\sum_{k=1}^n\oplus(X_k\oplus\ell^{\infty}(\Delta_k))\to(\sum_{n=1}^{\infty}\oplus(X_n\oplus\ell_{\infty}(\Delta_n))_{\infty})_{\infty}\]
such that

\begin{enumerate}
\item $\|i_n\|\leq C$ for every $n\in\mathbb{N}$. \item For
$\overrightarrow{x}\in\sum_{k=1}^n\oplus(X_k\oplus\ell_{\infty}(\Delta_k))$ we have that
$R_{[1,n]}(i_n(\overrightarrow{x}))=\overrightarrow{x}$, $R_{(n,\infty)}(i_n(\overrightarrow{x}))\in \sum_{k=n+1}^{\infty}\oplus\ell_{\infty}(\Delta_k)$ while for every $l\geq n+1$ we have that $i_l(R_{[1,l]}i_n(\overrightarrow{x}))=i_n(\overrightarrow{x})$.
\end{enumerate}
\item Setting
$Y_n=i_n[\sum_{k=1}^n\oplus(X_k\oplus\ell^{\infty}(\Delta_k))]$, the
union $\cup_n Y_n$ is dense in $\mathcal{Z}$.
\end{enumerate}

\end{definition}

In order to simplify the symbolisms, for $I$ interval of $\N$ we shall write $\sum_{k\in I}\oplus(X_k\oplus\ell_{\infty}(\Delta_k))$ instead of $(\sum_{k\in I}\oplus(X_k\oplus\ell_{\infty}(\Delta_k))_{\infty})_{\infty}$.

\subsection{The general construction }
We now present the basic ingredients of constructing this new type of spaces which is based on the Bourgain-Delbaen(BD)-
method of construction (see \cite{BD}).

Let $(X_n,\|\cdot\|_n)_{n\in\N}$ be a sequence of separable Banach spaces and let also $(\Delta_n)_{n\in\N}$ be a sequence of pairwise disjoint intervals of $\N$. We denote by $\Gamma$ the union $\Gamma=\cup_{n\in\N}\Delta_n$ and use letters as $\gamma,\xi,\eta$ to denote elements of the sets $\Gamma$.
For every $\gamma\in\Delta_n$ we assign a linear
functional
$c_{\gamma}^*:\sum_{k=1}^{n-1}(X_k\oplus\ell_{\infty}(\Delta_k))\to
\R$
and for $n<m\in\N$ we define by induction a linear operator
\[i_{n,m}:\sum_{k=1}^n\oplus
(X_k\oplus\ell_{\infty}(\Delta_k)) \rightarrow\sum_{k=1}^m\oplus
(X_k\oplus\ell_{\infty}(\Delta_k))\] as follows: For $m=n+1$ and
$\overrightarrow{x}\in\sum_{k=1}^n\oplus(X_k\oplus\ell^{\infty}(\Delta_k))$
\[i_{n,n+1}(\overrightarrow{x})=(\overrightarrow{x},\overrightarrow{x}_{n+1})\] where $\overrightarrow{x}_{n+1}=(0_{X_{n+1}},(c_{\gamma}^*(\overrightarrow{x})_{\gamma\in\Delta_{n+1}})\in X_{n+1}\oplus\ell_{\infty}(\Delta_{n+1})$.

Then assuming that $i_{n,m}$ has been defined, we set
$i_{n,m+1}=i_{m,m+1}\circ i_{n,m}$. It is clear that for every $n<l<m$,
$i_{n,m}=i_{l,m}\circ i_{n,l}$.

\begin{remark}\label{bp} \textbf{The boundeness principle for $(i_{n,m})_{n\leq m}$}.
Assume that there exists $C>0$ such that $\|i_{n,m}\|\leq C$ for every $n\leq m$. We define $i_n:\sum_{k=1}^n\oplus
(X_n\oplus\ell^{\infty}(\Delta_n))\to\ell_{\infty}[(X_n,\ell_1(\Delta_n))]$ as the direct limit \[i_n=\varinjlim_{m\rightarrow\infty} i_{n,m}.\] It follows that the operators $i_n$ are uniformly bounded by $C$ and since $\|i_n(\overrightarrow{x})\|_{\infty}\geq\|\overrightarrow{x}\|_{\infty}$ the operators $i_n$ are isomorphic embeddings. We can then define $Y_n=i_n[\sum_{k=1}^n\oplus (X_k\oplus\ell_{\infty}(\Delta_k))]$ and $\mathcal{Z}=\overline{\cup_{n\in\N}Y_n}$. It immediate follows that $Y_n\subset Y_{n+1}$, each $Y_n$ is C-isomorphic to $\sum_{k=1}^n\oplus (X_k\oplus\ell_{\infty}(\Delta_k))$ and moreover considering Definition \ref{mathcal} we conclude by the above that $\mathcal{Z}=(\sum_n\oplus X_n)_{BD}$.

Viewing $\mathcal{Z}$ as a close subspace of $(\sum_{n=1}^{\infty}\oplus(X_n\oplus\ell_{\infty}(\Delta_n))_{\infty})_{\infty}$ we restrict the operators $R_{[1,n]}:\mathcal{Z}\to\sum_{k=1}^n\oplus(X_k\oplus\ell^{\infty}(\Delta_k))$ for every $n\in\N$ and we also restrict the image of $i_n$ upon $\mathcal{Z}$, $i_n:\sum_{k=1}^n\oplus(X_k\oplus\ell^{\infty}(\Delta_k))\to\mathcal{Z}$. We can also extend $c_{\gamma}^*:\mathcal{Z}\to\R$ for every $\gamma\in\Delta_n$ as $c_{\gamma}^*(z)=c_{\gamma}^*(R_{[1,n]}z)$.

 We note that the definition of $i_n$ yields that for $x=i_n(\overrightarrow{u})\in Y_n$,  \[\|x\|=\max\{\|\overrightarrow{u}\|_{\infty}, \sup_{\gamma\in\Gamma}|c_{\gamma}^*(x)| \}.\]
\end{remark}

For each $n$, we define projections $P_{[1,n]}:\mathcal{Z} \to \mathcal{Z}$ by the rule
$P_{[1,n]}=i_n\circ R_{[1,n]}|_{\mathcal{Z}}$ with $Im P_{[1,n]}=Y_n$. An easy observation is the following.

\begin{lemma}\label{dual1}
Let $n\in\N$ and $x\in Y_n$. Suppose that there exists $m<n$ such that $P_{[1,m]}x=0$ then there exists $\overrightarrow{u}\in\sum_{k=m+1}^n\oplus(X_k\oplus\ell_{\infty}(\Delta_k))$ such that $x=i_n(\overrightarrow{0}_{1,m},\overrightarrow{u})$.
\end{lemma}
\bp We first treat (1). Note that $x=P_{[1,n]}(x)=i_n(\overrightarrow{y})$ where $\overrightarrow{y}=R_{[1,n]}(x)\in \sum_{k=1}^n\oplus(X_k\oplus\ell_{\infty}(\Delta_k))$. Since $P_{[1,m]}x=0_{\mathcal{Z}}=i_m(R_{[1,m]}x)$ it follows that $R_{[1,m]}x=\overrightarrow{0}_{1,m}$ which yields that $\overrightarrow{y}=(\overrightarrow{0}_{1,m},\overrightarrow{u})$ such that $\overrightarrow{u}=R_{[m+1,n]}(x)\in\sum_{k=m+1}^n\oplus(X_k\oplus\ell_{\infty}(\Delta_k))$.
\ep

\subsection{Schauder Decomposition of our space $\mathcal{Z}$}
For $n\leq m$ we define $P_{(n,m]}=i_mR_{[1,m]}|_{\mathcal{Z}}-i_nR_{[1,n]}|_{\mathcal{Z}}$. Since $i_n$ are uniformly bounded (property b) we have that $\|P_{I}\|<2C$ for every interval $I\subset\N$. Moreover, by property (d) of Definition \ref{mathcal} for every $n,m\in\N$ it is clear that $P_{[1,n]}\circ P_{[1,m]}=P_{[1,\min\{m,n\}]}$. Since $\mathcal{Z}=\overline{\cup_{n\in\N} P_{[1,n]}[\mathcal{Z}]}$, setting $Z_1=P_{[1,1]}$ and $Z_n=P_{\{n\}}[\mathcal{Z}]$, where $P_{\{n\}}=P_{[1,n]}-P_{[1,n-1]}$ it follows easily that $(Z_n)_{n\in\N}$ is a Schauder decomposition of the space $\mathcal{Z}$.

Note that for every $k\leq m$ an element $x$ belongs in $\sum_{n=k}^mZ_n$ if and only if $x\in Y_m$ and $P_{[1,k]}x=0$.
We finally note that by Lemma \ref{dual1} \[Z_n=\{(z=i_n(\overrightarrow{0}_{1,n},\overrightarrow{u}):\ \ \overrightarrow{u}\in X_n\oplus\ell_{\infty}(\Delta_n)\}.\]

Let $P_{[1,n]}^*:\mathcal{Z}^*\to\mathcal{Z}^*$ be the adjoint projections. We define $\mathcal{Z}_{[1,n]}^*=Im P_{[1,n]}^*$ for each $n\in\N$ and we observe the following.

\begin{lemma}\label{shrinking}
For every $n$, the restricted operator $i_n^*:\mathcal{Z}_{[1,n]}^*\to(\sum_{k=1}^n(X_k^*\oplus\ell_1(\Delta_k))_1)_1$ is an isomorphism onto satisfying that $\|z^*\|\leq\|i_n^*(z^*)\|_1\leq C\|z^*\|$ for every $z^*\in Z_{[1,n]}^*$.
\end{lemma}
\bp Since $\|i_n^*\|=\|i_n\|$ for every $n\in\N$ the right hand inequality is trivial. For the left hand inequality it is enough to show that for every $f\in \mathcal{Z}_{[1,n]}^*$ such that $\|f\|=1$ we have that $\|i_n^*(f)\|\geq 1-\e$ for every $\e>0$. Let $z^*\in \mathcal{Z}_{[1,n]}^*$ such that $\|z^*\|=1$ and $\e>0$. Let also $z\in\mathcal{Z}$ such that $\|z\|=1$, $|z^*(z)|\geq 1-\e$. Then, we have that $z^*=P_{[1,n]}^*u^*$ for $u^*\in\mathcal{Z}^*$ and we set $\overrightarrow{x}=R_{[1,n]}(z)$. A simple observation is that $u^*(i_n\overrightarrow{x})=u^*(i_n R_{[1,n]}z)=u^*P_{[1,n]}z=P_{[1,n]}^*u^*(z)=z^*(z)\geq 1-\e$. Therefore, $\|i_n^*(z^*)\|_1\geq|i_n^*z^*(\overrightarrow{x})|=|z^*(i_n\overrightarrow{x})|=|u^*P_{[1,n]}i_n\overrightarrow{x}|=|u^*(i_n\overrightarrow{x})|\geq 1-\e$. It remains to show that $i_n^*|_{Z_{[1,n]}^*}$ is onto. Observe that $i_n^*:\mathcal{Z}^*\to(\sum_{k=1}^n(X_n^*\oplus\ell_1(\Delta_n))_1)_1$ is onto since $i_n$ is isomorphic embedding  and thus for $\overrightarrow{f}\in(\sum_{k=1}^n(X_n^*\oplus\ell_1(\Delta_n))_1)_1$ there exists $g\in \mathcal{Z}^*$ such that $i_n^*(g)=\overrightarrow{f}$. Note that $i_n^*(g)=i_n^*(P_{[1,n]}^*g)$ and the result follows.
\ep

The next lemma yields that for every $l\geq n$ the restriction of the operator $i_l^*:\mathcal{Z}^*\to(\sum_{k=1}^l(X_k^*\oplus\ell_1(\Delta_k))_1)_1$ upon $\mathcal{Z}_{[1,n]}^*$ extends the operator $i_n^*:\mathcal{Z}_{[1,n]}^*\to(\sum_{k=1}^n(X_k^*\oplus\ell_1(\Delta_k))_1)_1$.

\begin{lemma}\label{dual2}
Let $n\in\N$ and $f\in Z_{[1,n]}^*$. Then, for every $l\geq n$ we have that $i_l^*(f)=(i_n^*(f),\overrightarrow{0}_{n+1,l})$.
\end{lemma}
\bp Let $g\in\mathcal{Z}^*$ such that $f=P_{[1,n]}^*g$. For $\overrightarrow{x}\in \sum_{k=1}^n\oplus X_k\oplus\ell_{\infty}(\Delta_k)$ we have that $i_l^*f(\overrightarrow{x},\overrightarrow{0}_{n+1,l})=f(i_l(\overrightarrow{x},\overrightarrow{0}_{n+1,l}))=
 P_{[1,n]}^*g(i_l(\overrightarrow{x},\overrightarrow{0}_{n+1,l}))=g(i_n\overrightarrow{x})=g(P_{[1,n]}i_n\overrightarrow{x})=i_n^*f(\overrightarrow{x})$.
  Moreover, if $\overrightarrow{x}\in\sum_{k=n+1}^l\oplus (X_k\oplus\ell_{\infty}(\Delta_k))$ then

\begin{eqnarray*}
i_l^*f(\overrightarrow{0}_{1,n},\overrightarrow{x})&=&
  f(i_l(\overrightarrow{0}_{1,n},\overrightarrow{x}))=
  (P_{[1,n]}^*g)(i_l(\overrightarrow{0}_{1,n},\overrightarrow{x},\overrightarrow{0}_{n+1,l})) \\ &=&g(i_n(\overrightarrow{0}_{1,n}))=0.
\end{eqnarray*}
\ep

\section{The dual of $(\sum_n\oplus X_n)_{BD}$ }
Before proceeding to the main construction we shall investigate the dual of BD $\mathcal{L}^{\infty}$ sums of a sequence of separable Banach spaces. We start by fixing such a sequence $(X_n,\|\cdot\|_n)_{n\in\N}$ and let $\mathcal{Z}=(\sum_n\oplus X_n)_{BD}$ satisfying Definition \ref{mathcal}. As we will see we have more things to say considering the dual in the case that the decomposition of the space $\mathcal{Z}$ is shrinking.

We need first to define the following operator \[\Phi:\overline{\cup_{n=1}^{\infty}Z_{[1,n]}^*}\to (\sum_{n=1}^{\infty}\oplus(X_n^*\oplus\ell_1(\Delta_n))_1)_1\] as follows:

 For $f\in
\mathcal{Z}^*$, we define $\Phi(f)=\lim_{n\to\infty}i_n^*P_{[1,n]}^*(f)$. Lemma \ref{dual2} and Lemma \ref{shrinking} yield that $\Phi$ is well defined and moreover the extended $\Phi:\overline{\cup_{n=1}^{\infty}Z_{[1,n]}^*}\to (\sum_{n=1}^{\infty}\oplus(X_n^*\oplus\ell_1(\Delta_n))_1)_1$ is an isomorphism onto.

\begin{proposition}\label{boundeness}
If the decomposition $(Z_n)_{n\in\N}$ of $\mathcal{Z}$ is shrinking, then $\mathcal{Z}^*$ is isomorphic with
$(\sum_{n=1}^{\infty}\oplus(X_n^*\oplus\ell_1(\Delta_n))_1)_1$.
\end{proposition}
\bp  We just observe that if the decomposition is shrinking then $\mathcal{Z}^*=\overline{\cup_{n=1}^{\infty}\mathcal{Z}_{[1,n]}^*}$ and thus the isomorphism $\Phi:\mathcal{Z}^*\to (\sum_{n=1}^{\infty}\oplus(X_n^*\oplus\ell_1(\Delta_n))_1)_1$ defined as above yields the result.
\ep

The following results concern spaces $\mathcal{Z}=(\sum_n\oplus X_n)_{BD}$ with a shrinking schauder decomposition and are helpful in terms of studying the operators acting on $\mathcal{Z}$. We recall that for $I$ interval of $\N$, $R_I^*$
 denotes the natural restriction $R_I^*:(\sum_{n=1}^{\infty}\oplus(X_n^*\oplus\ell_1(\Delta_n))_1)_1\to (\sum_{n\in I}\oplus(X_n^*\oplus\ell_1(\Delta_n))_1)_1$.

\begin{lemma}\label{cut}
Let $\e>0$, $f\in\mathcal{Z}^*$ and $m<n$. Let also $x\in Y_n$ such that $\|x\|\leq 1$, $P_{[1,m]}(x)=0$ and $|f(x)|\geq\e$. Then there exists $l\geq n$ and
$\overrightarrow{z}\in\sum_{k=m+1}^l\oplus (X_k\oplus\ell_{\infty}(\Delta_k))$ such that $\|\overrightarrow{z}\|_{\infty}\leq 1$ and
$|\Phi f(\overrightarrow{0}_{1,m},\overrightarrow{z},\overrightarrow{0}_{l+1})|\geq\frac{\e}{2}$. In particular,
$\|R_{[m+1,l]}^*(\Phi f)\|_1\geq\frac{\e}{2}$
\end{lemma}
\bp
Let $\delta>0$ such that $\delta<\frac{\e}{2(\|T\|+1)}$. Let $l\geq n$ and $g\in Z_{[1,l]}^*$ such that $\|f-g\|\leq \delta$. By Lemma
 \ref{dual1} we have that $x=i_l(\overrightarrow{0}_{1,m},\overrightarrow{z})$ where
 $\overrightarrow{z}\in (\sum_{k=m+1}^l\oplus (X_k\oplus\ell_{\infty}(\Delta_k))$. Observe that $\|\overrightarrow{z}\|_{\infty}\leq\|x\|\leq 1$
 and using Lemma \ref{dual2} we deduce that
 \[|\Phi g(\overrightarrow{0}_{1,m},\overrightarrow{z},\overrightarrow{0}_{l+1})|=|(i_l^*g)(\overrightarrow{0}_{1,m},\overrightarrow{z})|
 =|g(x)|\geq|f(x)|-\|f-g\|\geq \e-\delta.\] Thus, $|\Phi f(\overrightarrow{0}_{1,m},\overrightarrow{z},\overrightarrow{0}_{l+1})|\geq
 |\Phi g(\overrightarrow{0}_{1,m},\overrightarrow{z},\overrightarrow{0}_{l+1})|-\|\Phi\|\|f-g\|\geq \frac{\e}{2}$.
\ep

\begin{corollary}\label{blockdual}
Let $\e>0$, $(x_k)_{k\in\N}$ be a block sequence in $\mathcal{Z}$ and $f_k\in\mathcal{Z}^*$ such that $|f_k(x_k)|\geq \e$. Then there exists finite pairwise disjoint intervals of $\N$ such that $\|R_{I_k}^*(\Phi f_k)\|\geq\frac{\e}{2}$.
\end{corollary}
\bp
Since $(x_k)_{n\in\N}$ is block we can find $m_1<n_1<m_2<n_2<\ldots$ such that $x_k\in Y_{n_k}$ and  $P_{[1,m_k]}x_k=0$.
By Lemma \ref{cut} for every $k\in\N$ there exists $l_k\geq n_k$ such that setting $I_n=(m_k,l_n]$ we have that
\[\|R_{I_k}^*(\Phi f_k)\|\geq\frac{\e}{2}.\] Passing to a subsequence we can achieve that $I_k$ are pairwise disjoint and the proof is complete.
\ep

The following proposition states a property first noticed in the Schauder sum $(\sum_{k\in\N}\oplus(X_k^*\oplus\ell^1(\Delta_n))_1)_1$ but it will be presented in a more general manner.

\begin{proposition}\label{duality}
Let $(W_n)_{n\in\N}$ be a sequence of Banach spaces and let $W=(\sum_{n\in\N}\oplus W_n)_1$. Suppose that there exists a sequence $(\overrightarrow{w_k})_{k\in\N}$, $\e>0$ and a sequence $(I_k)_{n\in\N}$ of successive intervals of $\N$ such that $\|\overrightarrow{w_k}|_{\sum_{n\in I_k}W_n}\|\geq\e$ for every $k\in\N$. Then $(\overrightarrow{w_k})_{n\in\N}$ cannot be weakly null.
\end{proposition}
\bp
Suppose that $(\overrightarrow{w_k})_{k\in\N}$ is weakly null. We may also assume passing to subsequences and rearranging the numbering of $(I_k)_{k\in\N}$ that  $\overrightarrow{w_k}\in(\sum_{n=1}^{\max I_k}W_n)_1$ for every $k\in\N$. Let $k_1=1$ and choose $\overrightarrow{f_1}\in(\sum_{n\in I_1}\oplus W_n^*)_{\infty}$ such that $\|\overrightarrow{f_1}\|_{\infty}\leq 1$ and $\overrightarrow{f_1}(\overrightarrow{w_1})\geq\e$. Since $(\overrightarrow{w_k})_{k\in\N}$ is weakly null there exists $N_1\subset\N$ infinite such that $\min N_1=k_2>k_1$ and $\sum_{k\in N_1}|\overrightarrow{f_1}(\overrightarrow{w_k})| <\frac{\e}{2^2}$. Choose $\overrightarrow{f_2}\in (\sum_{n\in I_{k_2}}\oplus W_n^*)_1$ such that $\|x_2\|_{\infty}\leq 1$ and $\overrightarrow{f_2}(\overrightarrow{w_{k_2}}) \geq\e$. Following this manner inductively we find $1=k_1<k_2<\ldots$ and elements $(\overrightarrow{f_j})_{j\in\N}$ such that $\overrightarrow{f_j}\in (\sum_{n\in I_{k_j}}\oplus W_n^*)_{\infty}$, $\|\overrightarrow{f_j}\|_{\infty}\leq 1$, $\overrightarrow{f_j}(\overrightarrow{w_{k_j}}) >\e$ and $\sum_{i=j+1}^{\infty}\overrightarrow{f_j}(\overrightarrow{w_{k_i}}) <\frac{\e}{2^{j+1}}$. Setting $\overrightarrow{f}=\sum_{j=1}^{\infty}\overrightarrow{f_j}$ we have that $\|f\|_{\infty}\leq 1$ and since $\overrightarrow{w_{k_j}}\in(\sum_{n=1}^{\max I_{k_j}}W_n)_1$ the following is deduced
\begin{eqnarray*}
\overrightarrow{f}(\overrightarrow{w_{k_j}})&=&\sum_{i=1}^j\overrightarrow{f_i}(\overrightarrow{w_{k_j}}) \geq \overrightarrow{f_j}(\overrightarrow{w_{k_j}})-
\sum_{i\neq j}|\overrightarrow{f_i}(\overrightarrow{w_{k_j}})|\\ &\geq& \e-\sum_{i=1}^j\frac{\e}{2^{i+1}}\geq\frac{\e}{2},
\end{eqnarray*}
contradicting the fact that $(\overrightarrow{w_{k_j}})_{j\in\N}$ is weakly null.
\ep

\begin{definition}
For every element $z\in\mathcal{Z}$ we define the range of $z$, denoted by $ran(z)$ to be the
minimal interval $I\subset\N$ such that $x\in\sum_{n\in I}Z_n$. In a similar manner for $b^*\in\mathcal{Z}^*$ we say that $\ran b^*=(p,q)$ if $b^*(z)=P_{(p,q)}^*b^*(z)$ for every $z\in\mathcal{Z}$. For $z_1,z_2\in
\mathcal{Z}$, we will write $z_1<z_2$ whenever $\ran z_1<\ran z_2$ and
we will say that a sequence $(z_k)_{k\in I}$ is horizontally block in
$\mathcal{Z}$, if it is block with respect to the natural
decomposition $(Z_n)_{n\in\N}$, i.e. $z_k<z_{k+1}$ for every
$k\in I$. In the sequel, whenever we refer to a block sequence, we
mean that the sequence is horizontally block.
\end{definition}

The next result concerns $(\sum_n\oplus X_n)_{BD}$ where the separable spaces $X_n$ satisfy additional properties.

\begin{proposition}\label{blocking}
Let $\mathcal{Z}=(\sum_n\oplus X_n)_{BD}$ with a shrinking Schauder decomposition such that $\ell_1$ does not embed in
$X_n^*$ for every $n\in\N$ or $X_n$ admits the Schur property for every $n\in\N$. Let also $T:\mathcal{Z}\to\mathcal{Z}$ be a linear bounded operator.
Then, for every $(z_k)_{k\in\N}$ bounded block in $\mathcal{Z}$ and $q\in\N$ there exists a subsequence $(z_j)_{j\in\N}$ such that $\|P_{[1,q]}T(z_j)\|\to 0$.
\end{proposition}
\bp We treat first the case that $X_n$ has the Schur property for every $n\in\N$. Let $(z_k)_{k\in\N}$ be a bounded block in $\mathcal{Z}$.
Since $(Tz_k)_{n\in\N}$ is weakly null then for every $q\in\N$ $(P_{[1,q]}Tz_k)_{n\in\N}$ is weakly null. For each $i$ using a sliding hump argument and the Schur property of $X_i$ we deduce that there exists a subsequence $(x_j)_j$ such that $\|P_{\{i\}}Tx_j\|\to 0$ and the result follows.

In the case that $\ell_1$ does not embed in $X_n^*$ for every $n\in\N$ we assume on the contrary that there exists $q\in\N$ and $\delta>0$ such that
$\|P_{[1,q]}T(z_k)\|\geq\delta$ for every $k\in\N$. Applying Hahn-Banach we find $w_k^*\in \mathcal{Z}^*$ such that $|w_k^*(P_{[1,q]}T(z_k))|\geq\delta$.
Using the adjoint operator $T^*\circ P_{[1,q]}^*:\mathcal{Z}^*\to \mathcal{Z}^*$ and setting $z_k^*=T^*\circ P_{[1,q]}^*(w_k^*)$ it is clear that $|z_k^*(z_k)|\geq\delta$. We claim that there exists a subsequence of $(z_{n_k}^*)_{k\in\N}$ that is equivalent to the unit vector basis of $\ell_1$.

Indeed, if not by Rosenthal's $\ell_1$ theorem we may assume, passing to a subsequence, that $(z_k^*)_{k\in\N}$ is weakly Cauchy. Then, since $(z_k)_{k\in\N}$ is weakly null we can choose inductively $k_1<k_2<\ldots$ such that setting $\tilde{z}_i^*=z_{k_{2i}}^*-z_{k_{2i-1}}^*$ we have that $|\tilde{z}_i^*(z_{k_{2i}})|\geq\frac{\delta}{2}$. It is clear that $(\tilde{z}_i^*)_i$ is weakly null and applying Corollary \ref{blockdual} we find a sequence $(I_i)_{i\in\N}$ of successive intervals of $\N$ such that $\|R_{I_i}^*(\Phi(\tilde{z}_i^*))\|_1\geq \frac{\delta}{4}$. Proposition \ref{duality} implies that $\Phi(\tilde{z}_i^*)_{i}$ is not weakly null yielding a contradiction.

It follows that $(P_{[1,q]}^*(w_{n_k}^*))_k$ is equivalent to the unit vector basis of
$\ell_1$ and thus $\ell_1\in Z_{[1,q]}^*$. Since $Z_{[1,q]}^*\simeq(\sum_{n=1}^{q}(X_n^*\oplus\ell_1(\Delta_n))_1)_1$ there exists $n_0\in [1,q]$ such that $\ell_1$ is isomorphically embedded in $X_{n_0}^*$ yielding a contradiction.
\ep

\section{The Description of the Bourgain-Delbaen sets}
Now we shall give a basic description of how the BD- sets $\Delta_n$ will be determined as well as the functionals $c_{\gamma}^*:\sum_{k=1}^{n-1}\oplus (X_k\oplus\ell_{\infty}(\Delta_k))\to\R$ for $\gamma\in\Delta_n$.
We fix a sequence of separable Banach spaces $(X_n,\|\cdot\|_n)$ and two natural
numbers $a,b$ such that $0<a\leq 1$ and $0<b<\frac{1}{2}$. We determine the sets $\Delta_n$ by induction satisfying:
\begin{enumerate}
\item[(i)] Each $\Delta_n$ is finite and is the union of two pairwise disjoint sets $\Delta_n=\Delta_n^0\cup\Delta_n^1$.
\item[(ii)] Every element $\gamma$
in $\Delta_{n}^0$ is represented as $\gamma=(n,a,f,p,0)$ with $p\leq n$ for
$f$ functional in a finite subset of the unit ball $B_{X_p^*}$, while if $\gamma\in\Delta_{n}^1$, $\gamma$ is
represented as $\gamma=(n,b,\eta,p,\xi)$, where $\eta\in \Delta_p^1$
for $p<n-1$ and $\xi\in\Delta_k$ with $p\leq k\leq n-1$.
\end{enumerate}

For $\gamma\in\Delta_n$ we denote by $e_{\gamma}^*$ the usual vector element of $\ell_1(\Delta_n)$. We identify $X_n^*$ with $(X_n^*\oplus\{0\})_1$ and $\ell_1(\Delta_n)$ with $(\{0\}\oplus\ell_1(\Delta_n))_1$ via identifications $X_n^*\ni x^*\mapsto \overrightarrow{x}^*=(x^*,0)$ and respectively $e_{\gamma}^*\mapsto\overrightarrow{e_{\gamma}}^*=(0,e_{\gamma}^*)$. Note that $\|x^*\|_{X_n^*}=\|\overrightarrow{x}^*\|_1$ and $\|\overrightarrow{e_{\gamma}}^*\|_1=\|e_{\gamma}^*\|_1$.

For $m\leq n$ $\overrightarrow{f}\in(X_m^*\oplus\ell_1(\Delta_m))$ and $\overrightarrow{y}\in\sum_{k=1}^n\oplus(X_k\oplus\ell_{\infty}(\Delta_k))$ we denote by $\overrightarrow{f}(\overrightarrow{y})$ the natural action $\overrightarrow{f}(\overrightarrow{y}(m))$.

Let $\overrightarrow{x}\in\sum_{k=1}^{n}\oplus(X_k\oplus\ell_{\infty}(\Delta_k))$ and assume that for every $k\leq n$ the sets $\Delta_k$ have been determined as well as the operators $(i_{m,l})_{m\leq l\leq n}$. For $\gamma=(n+1,x^*,p,0)\in\Delta_{n+1}^0$ we define
    \[c_{\gamma}^*(\overrightarrow{x})=a \overrightarrow{x}^*(\overrightarrow{x}),\]
while for $\gamma=(n+1,b,\eta,p,\xi)\in\Delta_{n+1}^1$ we define \[c_{\gamma}^*(\overrightarrow{x})=a \overrightarrow{e_{\eta}}^*(\overrightarrow{x}) +b\overrightarrow{e_{\xi}}^*(\overrightarrow{x}-(i_{p,n-1} R_{[1,p]}\overrightarrow{x})).\]

As described above every element $\gamma\Delta_n$ is represented as $\gamma=(n,a,f,p,0)$ or $\gamma=(n,b,\eta,p,\xi)$. Each coordinate represents a difference characteristic that defines uniquely $\gamma$ and we shall use some of the them to define useful concepts. In particular we define the first coordinate as the rank of $\gamma$, i.e $\rank(\gamma)=n$ whenever $\gamma\in\Delta_n$. We also define the second coordinate as the weight of $\gamma$, $w(\gamma)=a$ or $b$.

\subsection{Argyros Haydon BD-sets}
We will now use the Argyros-Haydon version of "Bourgain Delbaen" sets as they where presented in \cite{AH}.
We shall follow the above basic scheme in a more general manner for the fixed sequence of separable Banach spaces $(X_n,\|\cdot\|_n)_{n\in\N}$.

We choose a pair of strictly increasing sequences $(m_j)_{j\in\N}$, $(n_j)_{j\in\N}$ of natural numbers such  $n_j\geq m_j^2$ satisfying the requirements of Assumption 2.3 in \cite{AH}.

For every $I=[m,n]$ finite interval of $\N$ we identify $(\sum_{k\in I}\oplus X_k^*)_1$ with $(\sum_{k\in I}\oplus (X_k^*\oplus\{0\})_1)_1$ via identification \[\overrightarrow{f}=(x_m^*,\ldots,x_n^*)\mapsto \overrightarrow{f}=(\overrightarrow{x}_m^*,\ldots,\overrightarrow{x}_n^*)\] Similarly we identify $(\sum_{k\in I}\oplus\ell_1(\Delta_k))_1$ with $(\sum_{k\in I}\oplus(\{0\}\oplus\ell_1(\Delta_k))_1)_1$, i.e. for $\overrightarrow{b}^*=(b_m^*,\ldots,b_n^*)$ such that $b_k^*=\sum_{\gamma\in\Delta_k}a_{\gamma}e_{\gamma}^*$ we identify $\overrightarrow{b}^*$ with $\overrightarrow{b}^*=(\overrightarrow{b}_m^*,\ldots,\overrightarrow{b}_n^*)$ where $\overrightarrow{b^*_k}=\sum_{\gamma\in\Delta_k}a_{\gamma}\overrightarrow{e^*_{\gamma}}$.

We choose $F_n$ 1-norming countable and symmetric subset of the unit ball of $X_n^*$ (recall that $X_n$ is separable). We denote by $F_n^m$ the first m-elements of $F_n$ and for every $p<n$ we define a subset $K_{n,p}$ of $(\sum_{k=p+1}^n\oplus X_k^*)_1$ defined as \[K_{n,p}=\{\overrightarrow{f}=(x^*_{p+1},\ldots,x_m^*):\ x_k^*\in F_k^n\cup\{0\},\ \sum_k\|x_k^*\|_{X_k^*}\leq 1\}.\]

As in \cite{AH} (Section 4) we choose a strictly increasing sequence of natural numbers $(N_n)_{n\in\N}$, we set $G_n=\{q\in\Q:\ q=\frac{k}{l}\ \ l \text{divides}\ N_n!\}$ and define a subset $B_{n,p}$ of $(\sum_{k=p+1}^n\oplus\ell_1(\Delta_k))_1$ as
\[B_{n,p}=\{\overrightarrow{b^*}=(b_1^*,\ldots,b_k^*):\ \ b_k^*=\sum_{\gamma\in\Delta_k}a_{\gamma}e_{\gamma}^*,\ a_{\gamma}\in G_n \sum_{\gamma\in\Gamma_n\setminus\Gamma_p}|a_{\gamma}|\leq 1 \}\]

Along with $\Delta_n$ we will also recursively define an injection $\sigma:\Delta_n\to\N$
such that $\min\{\sigma(\gamma):\gamma\in\Delta_n\}\geq\max\{\sigma(\gamma):\gamma\in\Delta_{n-1}\}$, hence $\sigma(\gamma)>n$ for all $\gamma\in\Delta_n$. Since $\Delta_n$ is finite this is possible.

Let $\Delta_1^0=\emptyset$, $\Delta_1=\Delta_1^1=\{1\}$ and $\sigma(1)=2$.
Assume that $\Delta_k$ and $\sigma:\Delta_k\to\R$ have been defined for
every $k\leq n$  we define $\Delta_{n+1}=\Delta_{n+1}^0\cup\Delta_{n+1}^1$ such that
\[\Delta_{n+1}^0=\{(n+1,p,\overrightarrow{f},0):\ \overrightarrow{f}\in K_{n,p}  \} \]

\begin{align*}
\Delta_{n+1}^1 &=\left\{(n+1,m_{2j},p,\overrightarrow{b^*}): \ \ 2j\leq n+1,\ \ p<n \ \ ,\overrightarrow{b^*}\in B_{n,p}\right\}\\ &\cup
\left\{(n+1,\eta,m_{2j},p,,\overrightarrow{b^*}): \ 2j\leq n+1,\ \ p<n,\ \eta\in \Delta_p^1,\
w(\eta) = m_{2j},\right.\\
&\left.\qquad\qquad\qquad\qquad\qquad\qquad\ ,\overrightarrow{b^*}\in B_{n,p}\right\}\\
&\cup\left\{(n+1,m_{2j-1},\overrightarrow{e^*_{\eta}}): \ \ 2j\leq n+2, \ \ \eta\in \Delta_{k},\ \ k\leq n,\right.\\
&\left.\qquad\qquad\qquad\qquad\qquad\qquad\qquad\qquad\qquad\ w(\eta)= m_{4i-2}>m_{2j-1}^2 \right\}\\
&\cup\left\{(n+1,\eta,m_{2j-1},p,\overrightarrow{e^*_{\xi}}): \ \ 2j\leq n+2,\ \ \eta\in \Delta_p^1,\ \ w(\eta) =
m_{2j-1},\ ,\right.\\
&\left.\qquad\qquad\qquad\qquad\ \ \xi\in
\Delta_k, \ p<k\leq n,\ \ w(\xi)=
m_{4\sigma(\eta)}\right\}.
\end{align*}

We define $\sigma:\Delta_{n+1}\to\N$ satisfying that $\min\{\sigma(\gamma):\gamma\in\Delta_{n+1}\}\geq\max\{\sigma(\gamma):\gamma\in\Delta_{n}\}$.
The functionals $c_{\gamma}^*$ for $\gamma\in\Delta_n$ are defined as in the above section but with some differences concerning the parameters $a,b$. In our case $a$ will always be equal to $1$ and thus $w(\gamma)=1$ for every $\gamma\in\Delta_n^0$, $n\in\N$, while $b=m_j$ whenever $\gamma\in\Delta_n^1$ and $w(\gamma)=m_j$.

\begin{remark}\label{norming}
1. We note that for every $p\leq n$ and $\overrightarrow{f}\in K_{p,n}$ we have that $\|\overrightarrow{f}\|_1\leq 1$. Indeed, let $\overrightarrow{f}=(x^*_{p+1},\ldots,x^*_n)$ such that $x_k^*\in F_{k}^m\cup\{0_{X_k^*}\}$ and $\sum_k\|x_k^*\|_{X_k^*}\leq 1$. Observe that $\|\overrightarrow{x_k^*}\|_1=\|x_k^*\|_{X_k^*}$ and since $\|\overrightarrow{f}\|_1=\sum_k\|\overrightarrow{x_k^*}\|_1$ the result follows. Similarly $\|\overrightarrow{b}^*\|_1\leq 1$ for every $\overrightarrow{b^*}\in B_{p,n}$.

2. Let $\overrightarrow{b}^*=(b^*_{m+1},\ldots,b_n^*)\in(\sum_{k=m+1}^n\oplus\ell_1(\Delta_k))_1$ such that $\|b^*\|_1\leq 1$ and $b_k^*=\sum_{\gamma\in\Delta_k}r_{\gamma}e_{\gamma}^*$ where $r_{\gamma}$ is rational for every $\gamma$. Then, there exists $q>n$ such that $\overrightarrow{b}^*\in B_{q,m}$.
Indeed, since the sequence $(N_j)_j$ is strictly increasing we have that there exists $q\geq n$ such that the maximum of all denominators of $r_{\gamma}$ is less than $N_q$ for every $\gamma\in\cup_{k=m+1}^n\Delta_k$. By the definition of $B_{q,m}$ the result follows.
\end{remark}

\section{Argyros Haydon $\mathcal{L}^{\infty}$ sums of Banach spaces $(\sum_n\oplus X_n)_{AH}$}
For a fixed sequence $(X_n,\|\cdot\|_n)_{n\in\N}$ of separable Banach spaces we denote by $(\sum\oplus X_n)_{AH}$ the BD $\mathcal{L}^{\infty}$ sum of $(X_n,\|\cdot\|_n)$ that is constructed following the steps stated in section 4.1.
In order to ensure that the spaces $(\sum\oplus X_n)_{AH}$ are well defined we need to check the boundness principle of the operators $i_{n,m}:\sum_{k=1}^n\oplus
(X_k\oplus\ell_{\infty}(\Delta_k)) \rightarrow\sum_{k=1}^m\oplus
(X_k\oplus\ell_{\infty}(\Delta_k))$.

\begin{proposition}\label{bound_prince}
$\|i_{n,m}\|\leq 2$ for every $n,m\in\N$ with $n\leq m$.
\end{proposition}
\bp We will prove it using induction on $m$ and for all $n\leq m$. For $m=1$, it
is trivial. Now assume that for some $m\in\N$ and some $n\leq m$ we have that for every $k\leq n$
and $l\leq m$, $\|i_{k,l}\|\leq 2$. In order to prove that
$\|i_{n,m+1}\|\leq 2$ it is enough to show that for $\overrightarrow{z}\in
\sum_{k=1}^{n}(X_k\oplus\ell^{\infty}(\Delta_k))$ with $\|\overrightarrow{z}\|\leq 1$, $|c_{\gamma}^*(i_{n,m}\overrightarrow{z})|\leq 2$
for every $\gamma\in\Delta_{m+1}$.

Let $\gamma\in\Delta_{m+1}^0$ of the form $\gamma=(m+1,\overrightarrow{f},p,0)$ where $\overrightarrow{f}\in K_{p,m}$ and by Remark \ref{norming} $\|\overrightarrow{f}\|_1\leq 1$. Thus,
$|c_{\gamma}^*(i_{n,m}\overrightarrow{z})|=|\overrightarrow{f}(\overrightarrow{z})|\leq 1$.

Now, let
$\gamma\in\Delta_{m+1}^1$ of the form $\gamma=(m+1,\eta,m_j,p,\overrightarrow{b^*})$ where $\overrightarrow{b^*}\in\ B_{p,m}$. By the definition of $c_{\gamma}^*$ we obtain that
\[c_{\gamma}^*(i_{n,m}\overrightarrow{z})=\overrightarrow{e_{\eta}}^*(i_{n,m}\overrightarrow{z})+\frac{1}{m_j}\overrightarrow{b^*}[i_{n,m}\overrightarrow{z}-i_pR_{[1,p]}
(i_{n,m}\overrightarrow{z})].\]
Observe that if $p\geq n$, then using the definition of $i_{n,m}$ we have that
$i_{n,m}\overrightarrow{z}=i_pR_{[1,p]}(i_{n,m}\overrightarrow{z})$ and by the inductive hypothesis we are done. In case that $p<n$ we have that $\overrightarrow{e_{\eta}}^*(i_{n,m}\overrightarrow{z})=\overrightarrow{e_{\eta}}^*(\overrightarrow{z})$ and
\[\overrightarrow{b^*}[i_{n,m}\overrightarrow{z}-i_pR_{[1,p]}(i_{n,m}\overrightarrow{z})]=
\overrightarrow{b^*}[i_{n,m}\overrightarrow{z}-i_{p,m}(R_{[1,p]}\overrightarrow{z})].\] Again by Remark \ref{norming} we have that $\|\overrightarrow{b^*}\|_1\leq 1$ and using our inductive
assumption we conclude that
\[c_{\gamma}^*(i_{n,m}\overrightarrow{z})\leq \|\overrightarrow{z}\|+\frac{1}{m_j}(\|i_{n,m}\overrightarrow{z}\|+\|i_{p,m}\overrightarrow{z}\|)\leq 2.\]
\ep

We deduce that $\|i_n\|\leq 2$ for every $n\in\N$ and by Remark \ref{bp} setting $\mathcal{Z}=\overline{\cup_nY_n}$ where $Y_n=i_n[\sum_{k=1}^n\oplus(X_k\oplus\ell_{\infty}(\Delta_k))]$ we have that $\mathcal{Z}=(\sum_n\oplus X_n)_{AH}$.

\begin{notation}\label{restrict}
As noted in section 2 we restrict the operators $R_{[1,n]}:\mathcal{Z}\to\sum_{k=1}^n\oplus(X_k\oplus\ell_{\infty}(\Delta_k))$. For $\gamma\in\Delta_{n+1}$ we extend $c_{\gamma}^*:\mathcal{Z}\to\R$ as $c_{\gamma}^*(z)=c_{\gamma}^*(R_{[1,n]}z)$. In the same manner we can extend every $\overrightarrow{f}:\sum_{k=1}^n\oplus(X_k\oplus\ell_{\infty}(\Delta_k))\to\R$ to a functional $\overrightarrow{f}:\mathcal{Z}\to\R$ (such that $\overrightarrow{f}(z)=\overrightarrow{f}(R_{[1,n]}z)$) ).

We naturally consider every vector element $\overrightarrow{f}\in(\sum_{k=m}^n\oplus(X_k^*\oplus\ell_1(\Delta_k))_1)_1$ as a bounded linear functional $\overrightarrow{f}:\sum_{k=1}^n\oplus(X_k\oplus\ell_{\infty}(\Delta_k))\to\R$ and thus by the above as a linear bounded functional $\overrightarrow{f}:\mathcal{Z}\to\R$.
This includes the functionals $\overrightarrow{e_{\gamma}}^*$ for $\gamma\in\Delta_n$, $\overrightarrow{x}^*$ for $x^*\in X_n^*$, $\overrightarrow{b^*}\in B_{p,n}$, $\overrightarrow{f}\in K_{n,p}$ where $p<n$. Moreover as it follows by Lemma \ref{norming} all these extended functionals belong in the unit ball of $\mathcal{Z}^*$.
\end{notation}

We recall that $P_{(m,n]}=i_m R_{[1,m]}-i_nR_{[1,n]}$ for every $m\leq n$.
For $\gamma\in\Delta_n$ we define $d_{\gamma}^*:\mathcal{Z}\to\R$ as $d_{\gamma}^*=\overrightarrow{e_{\gamma}^*}\circ P_{\{n\}}$. Considering Notation \ref{restrict} we obtain that $d_{\gamma}^*=\overrightarrow{e_{\gamma}^*}-c_{\gamma}^*$.
The boundness principle of $i_{k,m}$ yields that
$\|i_n\|\leq 2$, $\|P_{[1,n]}\|\leq 4$,  $\|P_{(n,\infty)}\|\leq 3$ for every $n\in\N$ while $\|d_{\gamma}^*\|\leq 3$ for every $\gamma\in\Gamma$.

The following proposition is crucial in order to estimate norms in $\mathcal{Z}$. The proof shares similar arguments as in \cite{AH} (Proposition 4.5).
\begin{proposition}\label{evalanal}
For every $\gamma\in\Delta_{n+1}$ there exists natural numbers $0<p_1<q_1<p_2<q_2<\ldots<p_a<q_a=n+1$ with $a\leq n_j$, elements $(\xi_i)_{i=1}^a$ with $\xi_a=\gamma$, $\xi_i\in\Delta_{q_i+1}$, $w(\xi_i)=w(\gamma)$ and functionals $\overrightarrow{b_i^*}\in B_{p_i,q_i}$
such that \[\overrightarrow{e_{\gamma}^*}=\sum_{i=1}^ad_{\xi_i}^*+\frac{1}{m_j}\sum_{i=1}^a\overrightarrow{b_i^*}\circ P_{(p_i,q_i]}\]
\end{proposition}

The sequence $\{p_i,q_i,\xi_i,b_i^*\}_{i=1}^a$ is called the evaluation analysis of $\gamma$.
Note that for $\gamma\in\Delta_{n}^0$ (i.e. $w(\gamma)=1$) the elements $\gamma$ admits the trivial evaluation analysis consisting of its singleton.

\begin{notation}\label{decompose}
By our already introduced terminology for every $p\leq n$ we identify  $(\sum_{k=p}^n\oplus X_k)_{\infty}$ with $(\sum_{k=p}^n\oplus (X_k\oplus\{0\})_{\infty})_{\infty}$ as \[\overrightarrow{x}=(x_p,\ldots,x_n)\to \overrightarrow{x}=(\overrightarrow{x}_p,\ldots, \overrightarrow{x}_n)\] where $\overrightarrow{x}_k=(x_k,0)\in (X_k\oplus\{0\})_{\infty}$. Similarly we identify $(\sum_{k=p}^n\oplus \ell_{\infty}(\Delta_k))_{\infty}$ with $(\sum_{k=p}^n\oplus(\{0\}\oplus\ell_{\infty}(\Delta_k))_{\infty})_{\infty}$.
\end{notation}

\begin{remark}\label{split}
Let $m\leq n$ and $z\in\mathcal{Z}$ with $\ran z=(m,n]$ (or equivalently $z\in Y_n$ and $P_{[1,m]}(z)=0$). By Lemma \ref{dual1} there exists $\overrightarrow{u}\in\sum_{k=m+1}^{n}\oplus(X_k\oplus\ell_{\infty}(\Delta_k))$ such that $z=i_n(\overrightarrow{u})$.
Using the above identification we can split $\overrightarrow{u}$ as $\overrightarrow{u}'+\overrightarrow{u}''$ such that $\overrightarrow{u}'\in(\sum_{k=m+1}^n\oplus X_k)_{\infty}$ while $\overrightarrow{u}''\in(\sum_{k=m+1}^n\oplus\ell_{\infty}(\Delta_k))_{\infty}$. Indeed if $\overrightarrow{u}(k)=(x_k,y_k)$ we set $\overrightarrow{u}'=(x_{m+1},\ldots,x_n)$ and $\overrightarrow{u}''=(y_{m+1},\ldots,y_n)$.
We also set $z'=i_n(\overrightarrow{0}_{1,m},\overrightarrow{u}')$ and $z''=i_n(\overrightarrow{0}_{1,m},\overrightarrow{u}'')$. Then it is easy to check that $z=z'+z''$. Moreover note that \[\|\overrightarrow{u}'\|_{\infty}=\sup\{|\overrightarrow{x}^*(\overrightarrow{u})|:\ x^*\in X_k^*,\ \ m+1\leq k\leq n\}\] while
\[\|\overrightarrow{u}''\|_{\infty}=\sup\{|\overrightarrow{e_{\gamma}}^*(\overrightarrow{u})|:\ \gamma\in\Delta_k,\ \ m+1\leq k\leq n\}.\]
Considering Notation \ref{restrict} for $x^*\in X_k^*$ or $\gamma\in\Delta_k$ with $m+1\leq k\leq n$ we have that $\overrightarrow{x}^*(z)=\overrightarrow{x}^*(u)=x^*(x_k)$ and similarly $\overrightarrow{e_{\gamma}}^*(z)=\overrightarrow{e_{\gamma}}^*(u)=e_{\gamma}^*(y_k)=y_k(\gamma)$. In order to simplify the symbolisms we shall denote $\overrightarrow{u}''$ by $(\overrightarrow{u}_{\gamma})_{\gamma\in\Gamma_n\setminus\Gamma_m}$ where $\overrightarrow{u}_{\gamma}=\overrightarrow{e_{\gamma}}^*(\overrightarrow{u})$.
\end{remark}

\begin{definition}
Let $z\in\mathcal{Z}$ with $\max\ran(z)=n$ and let $\overrightarrow{u}\in\sum_{k=1}^{n}\oplus(X_k\oplus\ell_{\infty}(\Delta_k))$ such that $z=i_n(\overrightarrow{u})$. We define $\supp_{loc}(z)=\overrightarrow{u}$. We say that the local support of $z$ has no weight if $(\overrightarrow{u}_{\gamma})_{\gamma\in\Gamma_n}=\overrightarrow{0}_{1,n}$ where $\overrightarrow{u}_{\gamma}=\overrightarrow{e_{\gamma}}^*(\overrightarrow{u})$ as in Remark \ref{split}.
\end{definition}

The concept of the next lemma is based on arguments of Lemma 7.2 in
\cite{AH}.

\bl\label{norm} Let $z\in \mathcal{Z}$ such that $\ran z\subset(p,q]$.
Then, there exists $\gamma\in\Gamma$ with $\rank(\gamma)>p$ such that
$|\overrightarrow{e_{\gamma}}^*(z)|>\frac{1}{2}\|z\|$ \el

\bp Let $\supp_{loc}(z)=\overrightarrow{u}$ where $\overrightarrow{u}=\overrightarrow{u}'+(\overrightarrow{u}_{\gamma})_{\gamma\in\Gamma_q\setminus\Gamma_p}$ as in Remark \ref{decompose}. Since $z=i_q(\overrightarrow{0}_{1,p},\overrightarrow{u})$ and $\|i_q\|\leq 2$ we have that
$\|\overrightarrow{u}\|_{\infty}\geq\frac{1}{2}\|z\|$. Note that \[\|\overrightarrow{u}\|_{\infty}=\max\{\|\overrightarrow{u}'\|_{\infty},\|(\overrightarrow{u}_{\gamma})_{\gamma}\|_{\infty}\}.\] In the case that $\|\overrightarrow{u}\|_{\infty}=\|(\overrightarrow{u}_{\gamma})_{\gamma}\|_{\infty}$ if we consider Remark \ref{split} we can find $p+1\leq k\leq q$ and $\gamma\in\Delta_k$ such that $|\overrightarrow{e_{\gamma}}^*(\overrightarrow{u})|=\|\overrightarrow{u}\|_{\infty}$. Hence
\[\overrightarrow{e_{\gamma}}^*(z)=\overrightarrow{e_{\gamma}}^*(\overrightarrow{u})\geq\frac{1}{2}\|x\|\] Otherwise $\|\overrightarrow{u}\|_{\infty}=\|\overrightarrow{u}'\|_{\infty}$ as in Remark \ref{split} and applying Hahn Banach theorem we can find $x^*\in B_{X_k^*}$ with $p+1\leq k\leq q$ such that $|\overrightarrow{x}^*(\overrightarrow{u})|=\|\overrightarrow{u}\|_{\infty}$. Since $F_k$ is 1-norming in
$B_{X_k^*}$, we may assume that $x^*\in F_k$ and let $l\geq p$ such that $x^*\in F_k^{l}$ and set $\overrightarrow{f}=(\overrightarrow{x}^*,\overrightarrow{0}_{k+1,l})$. Observe that the functional $\overrightarrow{f}$ belongs in $K_{l,k-1}$ and moreover $\overrightarrow{f}(z)=\overrightarrow{x}^*(\overrightarrow{u})$. Let
$\gamma\in\Delta_{l+1}$ of the form $\gamma=(l+1,\overrightarrow{f},n,0)$. Notice that $d_{\gamma}^*(z)=0$, thus $|\overrightarrow{e_{\gamma}}^*(z)|=|c_{\gamma}^*(z)|=|\overrightarrow{f}(z)|=\|\overrightarrow{u}\|\geq\frac{1}{2}\|z\|$ as desired. \ep

\subsection{Rapidly Increasing Sequences in $\mathcal{Z} =(\sum_n X_n)_{AH}$.}

For the sequel whenever we write $\mathcal{Z}$ we refer to $\mathcal{Z}=(\sum\oplus X_n)_{AH}$ for the fixed sequence of separable Banach spaces  $(X_n,\|\cdot\|_n)_{n\in\N}$.

We recall that $\mathcal{Z}$ admits a Schauder Decomposition $(Z_n)_{n\in\N}$ for which we define the range of elements of $\mathcal{Z}$ as well as block sequences in $\mathcal{Z}$. We shall define special types of block sequences that are useful in order to study the space as well as its bounded linear operators.

We start with the following lemma which concerns the bounded block sequences of $\mathcal{Z}$ in general.

\begin{lemma}\label{functional}
Let $(z_n)_{n\in I}$ be a block sequence with $\|z_n\|\leq C$ for every $n\in I$ and let $(a_n)_{n\in I}$ be a sequence of scalars. Then for every $\overrightarrow{f}\in K_{p,m}$ with $p>m$ there exists $n_0\in\N$ such that $|\overrightarrow{f}(\sum_na_nz_n)|\leq C|a_{n_0}|$.
\end{lemma}
\bp Let $\overrightarrow{f}\in K_{p,m}$ of the form $\overrightarrow{f}=(x^*_{p+1},\dots,x^*_m)$ where $x_k^*\in F_k^m\oplus\{0\}$ for every $p+1\leq k\leq m$ and $\sum_k\|x_k^*\|_{X_k^*}\leq 1$. We also recall that $\|\overrightarrow{x_k^*}\|\leq\|\overrightarrow{x_k^*}\|_1=\|x_k^*\|_{X_k^*}$. Since $(z_n)_n$ is block for every $k$ there exists $n_k$ not necessarily pairwise different such that $\overrightarrow{x_k^*}(\sum_na_nz_n)=\overrightarrow{x_k^*}(a_{n_k}z_{n_k})$. We set $a_{n_0}=\max\{a_{n_k}:\ k=p+1,\ldots,m\}$. It follows that $|\overrightarrow{f}(\sum_ka_kz_k)|\leq C|a_{k_0}|\sum_k\|\overrightarrow{x_k^*}\|\leq C|a_{k_0}|$.
\ep
We continue with the definition of Rapidly Increasing Sequences (RIS).

\begin{definition}
We say that a block sequence $(z_k)_{k\in\N}$ in $\mathcal{Z}$ is $C$-RIS if there exists a constant $C>0$ and a strictly increasing sequence $(j_k)_{k\in\N}$ such that
\begin{enumerate}
\item $\|z_k\|\leq C$ for all $k\in\N$
\item $j_{k+1}>\max\ran z_k$
\item $\overrightarrow{|e_{\gamma}^*}(z_k)|\leq \frac{C}{m_i}$ whenever $w(\gamma)=m_i$ and $i<j_k$.
\end{enumerate}
\end{definition}

\begin{lemma}\label{RISX}
Let $z\in\mathcal{Z}$ such that $\max ran z=q$ and suppose that $\supp_{loc}(z)$ has no weight. Then, for every $\gamma\in\Gamma$ of $w(\gamma)=m_j$ we have that $|\overrightarrow{e_{\gamma}^*}(z)|\leq\frac{3\|z\|}{m_j}$.
\end{lemma}
\bp The proof share a similar concept of Lemma 5.7 in \cite{AH}. Observe that for every $\gamma\in\Gamma$ with $\rank\gamma>q$ we have that $d_{\gamma}^*(x)=0$ and $P_{(r,\infty)}(x)=0$ for every $r>q$. If $\gamma\in\Gamma$ with $w(\gamma)=m_j$ there exists $\{p_i,q_i,b_i^*,\xi_i\}_{i=1}^l$ , $l\leq n_j$ such that $\overrightarrow{e_{\gamma}}^*=\sum_{i}d_{\eta_i}^*+\frac{1}{m_j}\sum_i \overrightarrow{b_i}^*\circ P_{(p_i,q_i]}$. Let $i$ be the maximum such that $p_i\leq q$. For every $i'>i$ note that $p_{i'}>q$ and also $q_{i'}>p_{i'}>q$ hence $P_{(p_{i'},q_{i'}]}(z)=0$ and $d_{\xi_{i'}}^*(z)=0$. If $i=1$ then $\overrightarrow{e_{\gamma}}^*(z)=\frac{1}{m_j}\overrightarrow{b_1}^*(P_{(p_1,q_1]}z)$. Otherwise $i-1\geq 1$ and $\overrightarrow{e_{\gamma}}^*(z)=\overrightarrow{e_{\eta}}_{i-1}^*(z)+\frac{1}{m_j}\overrightarrow{b_i}^*(P_{(p_i,q_i]}z)$. Notice that if $\overrightarrow{u}=\supp_{loc}(z)$ then $\overrightarrow{e_{\eta}}_{i-1}^*(z)=\overrightarrow{e_{\eta}}_{i-1}^*(\overrightarrow{u})=0$ as $\rank(\eta_{i-1})=q_{i-1}+1<p_i<q$. Therefore in every case we conclude that $\overrightarrow{e_{\gamma}}^*(z)=\frac{1}{m_j}\overrightarrow{b_i}^*(P_{(p_i,q_i]}z)=\frac{1}{m_j}\overrightarrow{b_i}^*(P_{(p_i,\infty)}z)\leq \frac{3\|z\|}{m_j}$.
\ep

\begin{corollary}\label{RISXS}
Let $(z_k)_{k\in\N}$ be a block sequence in $\mathcal{Z}$ that is bounded by a constant $C>0$ and assume that $\supp_{loc}(z_k)$ has no weight for every $k\in\N$. Then there exists a subsequence $(z_i)_i$ that is $3C$-RIS.
\end{corollary}
\bp
Let $p_k,q_k$ such that $\ran z_k=(p_k,q_k]$ and let $\overrightarrow{u}_k=\supp_{loc}(z_k)$ in $\sum_{i=p_k}^{q_k}\oplus(X_i\oplus\ell_{\infty}(\Delta_i))$. We also assume passing to a subsequence that $p_{k+1}>q_k+1$ for every $k$. We set $j_k=p_k$ and then we have that $(j_k)_k$ is strictly increasing and $j_{k+1}>\max\ran x_k$. Moreover, for $\gamma\in\Gamma$ with $w(\gamma)=m_j$ by Lemma \ref{RISX} we deduce that $|\overrightarrow{e_{\gamma}}^*(z_k)|\leq 3C/m_j$. Since $\|z_k\|\leq C\leq 3C$, every one of the three conditions of an RIS are satisfied.
\ep

The following proposition ensures that there is a strong connection between RIS of $\mathcal{Z}$ and the mixed Tsirelson space
$T(\mathcal{A}_{3n_j},\frac{1}{m_j})_{j\in\N}$. It is proved in a similar manner as in \cite{AH} (Proposition 5.4) and we can use the same estimates. We denote by $(e_n)_{n\in\N}$ the usual basis of $T(\mathcal{A}_{3n_j},\frac{1}{m_j})_{j\in\N}$.

Considering its norming set $W=W[(\mathcal{A}_{3n_j},\frac{1}{m_j})_{j\in\N}]$ as a subset of $c_{00}$,
for every $f\in W$ we define $\supp f= \{k\in \N:\ \ f(k)\neq 0\}$ and whenever $f$ is of the form
$f=\frac{1}{m_j}\sum_{k=1}^{n}f_i$ for some $(f_i)_{i=1}^n\subset W$ we define the weight of $f$ as $weight(f)=m_j$.

\begin{proposition}\label{basiceq}
Let $(z_k)_{k\in I}$ be an C-RIS in $\mathcal{Z}$ and $\gamma\in\Gamma$. Then, for every scalars $a_k$ and $s\in\N$ there exist $k_0\in I$ and a functional
$f\in W[(\mathcal{A}_{3n_j},\frac{1}{m_j})_{j\in\N}]$  such that
\begin{enumerate}
\item Either $f=0$ or $w(\gamma)=weight (f)$, $\supp f \subset\{k\in I:\ k>k_0\}$.
\item $|\overrightarrow{e_{\gamma}^*}\circ P_{(s,\infty)}(\sum_{k=1}^{\infty}a_kz_k)|\leq 4C|a_{k_0}|+6Cf(\sum_{k=1}^{\infty}|a_k|e_k)$
\end{enumerate} Moreover, if we assume that there exists $j_0\in\N$ such that \[|\overrightarrow{e_{\xi}}^*(\sum_{k\in I}^{\infty}a_kz_k)|\leq C\max_{k\in I}|a_k|,\] for every $J\subset I$ and all $\xi\in\Gamma$ with $w(\xi)=m_{j_0}$, then $f$ can be chosen in $ W[(\mathcal{A}_{3n_j},\frac{1}{m_j})_{j\neq j_0}]$.
\end{proposition}
\bp
Suppose that $\gamma$ belongs to $\Delta_n^0$ for some $n\in\N$ of the form $\gamma=(n,\overrightarrow{g},p,0)$ where $p<n$, $\overrightarrow{g}\in K_{p,n-1}$. Then, $\overrightarrow{e_{\gamma}}^*=d_{\gamma}^*+c_{\gamma}^*$ where $d_{\gamma}^*=\overrightarrow{e_{\gamma}}^*\circ P_{\{n\}}$ and $c_{\gamma}^*=\overrightarrow{g}$. Note that if $s>n$ then $\overrightarrow{e_{\gamma}}^*\circ P_{(s,\infty)}=0$ and there is nothing to prove. Let $s\leq n$ and $w_k=P_{(s,\infty)}z_k$ for every $k$. Observe that since $(w_k)_k$ is block by Lemma \ref{functional} there exists $k_1\in I$ such that $c_{\gamma}^*(\sum_{k\in I}a_k w_k)\leq C|a_{k_1}|$. Moreover there exists $k_2\in I$ not necessarily equal to $k_1$ such that $d_{\gamma}^*(\sum_{k\in I}a_k w_k)=d_{\gamma}^*(a_{k_2}w_{k_2})$. Let $k_0\in I$ $(k_0=k_1$ or $k_0=k_2$) such that $|a_{k_0}|=\max\{|a_{k_1}|,|a_{k_2}|\}$. Since $\|P_{(s,\infty)}\|\leq 3$, $\|d_{\gamma}^*\|\leq 3$ we conclude that
\[|\overrightarrow{e_{\gamma}}^*\circ P_{(s,\infty)}(\sum_{k\in I}a_kz_k)|\leq 4C|a_{k_0}|.\]
Setting $f=0$ the result follows.
For $\gamma\in\cup_n\Delta_n^1$ we use induction on $\rank(\gamma)=n$ in a similar manner as in \cite{AH} (Proposition 5.4).
\ep

The basic inequality yields the following:
\begin{corollary}\label{basiceqcor}
Let $(z_k)_{k=1}^{n_{j_0}}$ be an C-RIS in $\mathcal{Z}$. Then $\|n_{j_0}^{-1}\sum_{k=1}^{n_{j_0}}z_k\|\leq \frac{10C}{m_{j_0}}$. Moreover, if $(\lambda_k)_{k=1}^{n_{j_0}}$ are scalars such that $|\lambda_k|\leq 1$ and \[|\overrightarrow{e_{\gamma}}^*(\sum_{k\in J}\lambda_kz_k)|\leq C\max_{k\in J}|\lambda_k|\] for every $\gamma$ of weight $m_{j_0}$ and every interval $J\subset\{1,2,\ldots,n_{j_0}\}$, then \[\|n_{j_0}^{-1}\sum_k\lambda_k z_k\|\leq\frac{10C}{m_{j_0}^2}.\]
 \end{corollary}
\bp We apply the basic inequality for scalars $n_{j_0}^{-1}\lambda_k$  and $I=\{1,2,\ldots,n_{j_0}\}$. Using estimations of $W(\mathcal{A}_{3n_j},\frac{1}{m_j})_{j\in\N}$ ( see \cite{AH}, Section 2.4) we conclude that $|\overrightarrow{e_{\gamma}}^*(n_{j_0}^{-1}\sum_{k=1}^{n_{j_0}}z_k)|\leq \frac{10C}{m_{j_0}}$ for every $\gamma\in\Gamma$. Observe that for $f\in B_{X_l^*}$ there exists at most one $k_l\in I$ such that  $|\overrightarrow{f}(n_{j_0}^{-1}\sum_{k=1}^{n_{j_0}}z_k)|=|\overrightarrow{f}(n_{j_0}^{-1}z_{k_l})|\leq \frac{C}{n_{j_0}}$. Since $n_{j_0}\geq m_{j_0}^2$ combining all the above the proof of the first part is complete.
For the second we can apply the "moreover part" of the basic inequality and using estimations of $W(\mathcal{A}_{3n_j},\frac{1}{m_j})_{j\neq j_0}$ we deduce that $|\overrightarrow{e_{\gamma}}^*(n_j^{-1}\sum_k\lambda_k z_k)|\leq \frac{10C}{m_{j_0}^2}$ for every $\gamma\in\Gamma$. In a similar manner as above we arrive to the desired estimation of the norm.
\ep

All the above yield to the following general result that concerns AH-$\mathcal{L}_{\infty}$ sums of separable Banach spaces.
\begin{proposition}\label{RISblocking}
Let $\mathcal{Z}=(\sum_n\oplus X_n)_{AH}$ where $(X_n,\|\cdot\|_n)_{n\in\N}$ is a sequence of separable Banach spaces. Let also $Y$ be a Banach space and $T:\mathcal{Z}\to Y$ be a linear and bounded operator such that $\|Tz_n\|\to 0$ for every RIS $(z_k)_{k\in\N}$ in $\mathcal{Z}$, Then $\|Tw_k\|\to 0$, for every bounded block sequence $(w_k)_{k\in\N}$ in $\mathcal{Z}$.
\end{proposition}

\bp For (1) first we fix a bounded block sequence $(z_k)_{k\in\N}$ in $\mathcal{Z}$ and let $C>0$ such that $\|z_k\|\leq C$ for every $k\in\N$. It is enough to show that there exists a subsequence $(z_{k_i})_i$ such that $\|T(z_{k_i})\|\to 0$.  Let $p_k,q_k$ such that $\ran z_k=(p_k,q_k]$ and let $\overrightarrow{u}_k=\supp_{loc}(z_k)$. We split each element $z_k$ to $z_k=z'_k+z''_k$ with $\ran z'_k=\ran z''_k=\ran z_k$ as in Remark \ref{split}. It follows that both $(z'_k)_{n\in\N}$ and $(z''_k)_{k\in\N}$ are bounded and moreover by the definition of $z'_k$ we have that $\supp_{loc}(z'_k)$ has no weight. By Corollary \ref{RISXS} there exists a subsequence $(z'_k)_{k\in M}$ that is $3C$-RIS, hence by our hypothesis $(T(z'_k))_{k\in M}$ is norm null. For every $k\in M$ and $N\in\N$ we also split further the element $z''_k=w_k^N+y_k^N$ in a similar manner as in Proposition 5.11 in \cite{AH}. Namely, we define $\overrightarrow{w}_k^N,\overrightarrow{y}_k^N\in(\sum_{i=p_k+1}^{q_k}\oplus\ell_{\infty}(\Delta_i))_{\infty}$ such that $\overrightarrow{w}_k^N(\gamma)=\overrightarrow{u}_{\gamma}$ if $w(\gamma)\leq m_N$ or 0 otherwise , while $\overrightarrow{y}_k^N(\gamma)=\overrightarrow{u}_{\gamma}$ if $w(\gamma)> m_N$ or 0 otherwise. It follows that $\overrightarrow{w}_k^N+\overrightarrow{y}_k^N=(\overrightarrow{u}_{\gamma})_{\gamma}$ and we set $w_k^N=i_{q_k}(\overrightarrow{0}_{1,p_k},\overrightarrow{w}_k^N)$ and $y_k^N=i_{q_k}(\overrightarrow{0}_{1,p_k},\overrightarrow{y}_k^N)$. For the bounded block sequences $(w_k^N)_{k\in M}$, $(y_k^N)_{k\in M}$ we find subsequences $(w_{k_{N_j}}^{N_j})_j$, $(y_{k_{N_j}}^{N_j})_j$ that are RIS following a similar argument as in \cite{AH} (Proposition 5.11). Our hypothesis implies that $\|T(z''_{k_{N_j}})\|\to 0$ and since $\|T(z'_{k_{N_j}})\|\to 0$ we conclude that $\|T(z_{k_{N_j}})\|\to 0$.
\ep

\begin{corollary}
Let $\mathcal{Z}=(\sum_n\oplus X_n)_{AH}$ where $(X_n,\|\cdot\|_n)_{n\in\N}$ is a sequence of separable Banach spaces. Then the dual $\mathcal{Z}^*$ may be identified with $(\sum_{n=1}^{\infty}\oplus
  (X_n^*\oplus\ell_1(\Delta_n))_1)_1$.
\end{corollary}
\bp By Corollary \ref{basiceqcor} we observe that every RIS is weakly null. By Proposition \ref{RISblocking} we conclude that every bounded block sequence
in $\mathcal{Z}$ is weakly null and thus the
  decomposition $(Z_n)_{n\in\N}$ is shrinking. Proposition \ref{boundeness} yields the result.
\ep

\section{The HI-property in block Subspaces of $\mathcal{Z}=(\sum_nX_n)_{AH}$}

In this section we will define the basic features that can be found in many HI constructions ( see \cite{AH}, \cite{AF}, also \cite{argetall}). As we already noted we follow the HI method of construction of Argyros-Haydon presented in \cite{AH}. The adaptation of the arguments is made without validating their precise estimations and we can observe by the definition of the sets $\Delta_n$ that in the case where $X_n=\{0\}$ for every $n\in\N$ the space $\mathcal{Z}=(\sum\oplus X_n)_{AH}$ coincides with the Argyros Haydon space $\mathfrak{X}_K$. Therefore for any arbitrary choice $(X_n,\|\cdot\|_n)_{n\in\N}$, the space $\mathcal{Z}$ will always contain the space $\mathfrak{X}_K$ and thus a $\mathcal{L}_{\infty}$ HI subspace. However as we will see the HI property is satisfied in every block subspace of $\mathcal{Z}$ which reveals the the influence of an HI external norm (see \cite{AF}, \cite{AR}).

We start by recalling the definition of Hereditarily indecomposable (HI) spaces.
\begin{definition}
We say that a Banach space $X$ is Hereditarily Indecomposable if  every closed subspace $Y$ of $X$ is indecomposable i.e. there do not exists $W,Z$ infinite dimensional closed subspaces of $Y$ such that $Y=W\oplus Z$.
\end{definition}
It is known (see \cite{AF},\cite{argetall},\cite{AH},\cite{AT}) that a Banach space $X$ is HI if and only if for every pair of infinite dimensional closed subspaces $Y,Z$ of $X$ and every $\e>0$ there exist $y\in Y$, $z\in Z$ such that $\|y+z\|>1$ while $\|y-z\|<\e$.

We continue by giving the definition of
$C-\ell_1^n$-averages.
\begin{definition}
Let $C>1$ and $n\in\N$. We say that a vector $z\in \mathcal{Z}$ is a
$C-\ell_1^n$ average if
\begin{enumerate}
\item $\|z\|\geq 1$ \item There exists a block sequence
$(z_i)_{i=1}^n$ in $Z$, with $\|z_i\|\leq C$, for all
$i=1,2,\ldots,n$ such that $z=\frac{1}{n}\sum_{i=1}^nz_i$.
\end{enumerate}
\end{definition}

The proof of the existence of $\ell^1$ averages in $\mathcal{Z}$ demands some a further study of the space.

\bpr\label{upper} Let $(z_k)_{n\in\N}$ be a bounded block sequence in $\mathcal{Z}$. Then, for every $j\in\N$ there exists $(z_{k_i})_{i=1}^{n_{2j}}$ such that
$2j\leq\max\ran z_{k_1}$ and an element $\gamma\in\Gamma$ of weight
$w(\gamma)=m_{2j}$ such that
\[|\overrightarrow{e^*_{\gamma}}(\sum_{i=1}^{n_{2j}}z_{k_i})|\geq\frac{1}{2m_{2j}}\sum_{k=1}^{n_{2j}}\|z_{k_i}\|\]
\epr

\bp Fix $j\in\N$ and let $p_k<l_k< p_{k+1}<l_{k+1}<\ldots$ such that $\ran z_k\subset(p_k,l_k]$.
Using Lemma \ref{norm} for every $k$, we can find
$\xi_k\in\Delta_{q_k}$ with $q_k\geq p_k+1$ such that
$\overrightarrow{e_{\xi_k}}^*(z_k)\geq\frac{1}{2}\|z_k\|$. We set $\overrightarrow{b}_k^*=(\overrightarrow{0}_{p_k+1,q_k-1},\overrightarrow{e_{\xi_k}}^*)$ and observe that $b_k^*\in B_{q_k,p_k}$ and $P_{(p_k,q_k]}^*\overrightarrow{b}_k^*(z_k)=\overrightarrow{e_{\xi_k}}^*(z_k)$ for all $k$. Let $k_1\in\N$ such that $2j\leq\max\ran z_{k_1}$ and $\eta_{k_1}=(q_{k_1}+1, m_{2j},\overrightarrow{b}_{k_1}^*)$. Assume that for some $1<i<n_{2j}$ the elements $z_{k_l}$, $\eta_{k_l}$ have been defined for every $l<i$. We choose $k_i$ such that $\rank\eta_{k_{i-1}} <p_{k_i}$ and let $\eta_{k_i}=(q_{k_i}+1,\eta_{k_i-1},m_{2j},\overrightarrow{b}_{k_i}^*)$. Observe that
that $d_{\eta_{k_i}}^*(z_{k_l})=0$ for every $i,l$. Let $\gamma\in\Gamma$ with $w(\gamma)=m_{2j}$ and evaluation analysis
$\{p_{k_i},q_{k_i},\eta_{k_i},\overrightarrow{b}_{k_i}^*\}_{i=1}^{n_{2j}}$. An immediate computation yields that
$\overrightarrow{e^*_{\gamma}}(\sum_{i=1}^{n_{2j}}z_{k_i})=
\frac{1}{m_{2j}}\sum_{i=1}^{n_{2j}}\overrightarrow{e_{\xi}}_{k_i}^*(z_{k_i})$ and hence $\gamma$ satisfies the conclusion.
\ep

As in \cite{AH} (Lemma 8.2), the above result in conjunction with a standard argument presented in Lemma 2.2 of \cite{AH} yields the following:
\begin{lemma}\label{average}
Let $Y$ be a block subspace of $\mathcal{Z}$. Then for every $C>1$ and $n\in\N$, $Y$ contains
$C-\ell_1^n$ average.
\end{lemma}

Next we generalise the result of Proposition \ref{upper} concerning weakly null sequences as it will be useful in next section.
\begin{lemma}\label{wnull}
Let $(z_k)_k$ be weakly null sequence in $\mathcal{Z}$ and assume that there exists a sequence of successive intervals of $\N$, $(J_k)_{k}$ such that $\|P_{J_k}(z_k)\|\geq \delta$. Then for every $j\in\N$ there exist elements $(z_{k_i})_{i=1}^{n_{2j}}$ with
$2j\leq\max\ran z_{k_1}$ and $\gamma\in\Gamma$ of weight
$m_{2j}$ such that
$|\overrightarrow{e^*_{\gamma}}(\sum_{i=1}^{n_{2j}}z_{k_i})|\geq\frac{\delta}{4m_{2j}}$.
\end{lemma}
\bp
Let $j\in\N$ and assume that $\ran z_k\subset[1,l_k]$ where $l_k\geq\max J_k$ for every $k\in\N$. Let $\xi_k\in\Delta_{q_k}$ with $q_k\geq \min J_k$ (Lemma \ref{norm}) such that
$\overrightarrow{e_{\xi_k}}^*(P_{J_k}z_k)\geq\frac{\delta}{2}$. We set $p_k=\min J_k-1$  and $\overrightarrow{b}_k^*=(\overrightarrow{0}_{p_k+1,q_k-1},\overrightarrow{e_{\xi_k}}^*)\in B_{q_k,p_k}$ for every $k$. We choose inductively $z_{k_i}$, $\eta_{k_i}$ as in the proof of Proposition \ref{upper} and we additionally require in each inductive step $r=i+1>1$ that the element $z_{k_r}$ satisfies that
$|\overrightarrow{e_{\eta}}_{k_i}^*(z_{k_r})|<\frac{\delta}{4m_{2j}}$. Since $(z_k)_k$ is weakly null such a choice is possible. Let $\gamma\in\Gamma$ with $w(\gamma)=m_{2j}$ and evaluation analysis
$\{p_{k_i},q_{k_i},\eta_{k_i},\overrightarrow{b}_{k_i}^*\}_{i=1}^{n_{2j}}$. A simple observation is that for each $1\leq i\leq n_{2j}$,  $\overrightarrow{e_{\gamma}}^*(z_{k_i})=\overrightarrow{e_{\eta}}_{k_{i-1}}^*(z_{k_i})+\frac{1}{m_{2j}}\overrightarrow{b}^*_{k_i}(P_{J_{k_i}}z_{k_i})
\geq\frac{\delta}{4m_{2j}}$. We conclude that
\[|\overrightarrow{e^*_{\gamma}}(\sum_{i=1}^{n_{2j}}z_{k_i})|\geq\frac{\delta}{4m_{2j}},\] as promised.
\ep

We recall the definition of exact pairs.
\begin{definition}\label{40}
Let $C>0$, $\e\in\{0,1\}$, and $j\in\N$. A pair
$(z,\gamma)\in \mathcal{Z}\times \Gamma$
is said to be $(C,j,\e)$ exact pair if the following
conditions are fulfilled:
\begin{enumerate}
\item $w(\gamma)=m_j$, $\|z\|\leq
C$, $\overrightarrow{e_{\gamma}}^*(z)=\e$.
\item $d_{\xi}^*(z)\leq\frac{C}{m_j}$ for every $\xi\in\Gamma$.
 \item If
$\gamma'\in\Gamma$ with $w(\gamma')=m_i$ and $i\neq j$,
then \[|\overrightarrow{e_{\gamma'}}^*(z)|\leq\begin{cases} Cm_i^{-1}\ \ \text{ if}\
i<j \\ Cm_j^{-1}\ \text{ if}\ i>j.\end{cases}\]
\end{enumerate}
\end{definition}

The next results of this section are similar to corresponding ones in \cite{AH}. We shall include slight description of the basic steps followed in the proofs for sake of completeness. We start with the following lemma that shares the same arguments with Proposition 8.6 in \cite{AH}.

\begin{lemma}\label{exactness} Let $Y$ be a block subspace of $\mathcal{Z}$. Then, for every $j\in\N$ there exists a $(65,2j,1)$ exact pair $(z,\eta)$ in $Y$.
\end{lemma}
\bp Let $(j_k)_{k\in\N}$ be an increasing sequence of natural numbers and $C>1$. Lemma \ref{average} implies that for each $k\in\N$ there exists a
C-$\ell_1^{n_{j_k}}$ average $z_k$ in $Y$.
The corresponding analogue of Lemma 8.4 in \cite{AH} allows as to assume that $(z_k)_{k\in\N}$ is $2C$-RIS (passing to a subsequence). We note that $\|z_k\|\geq 1$ for every $k\in\N$ and also for $j\in\N$ Proposition \ref{upper} yields that there exists a subsequence denoted by $(z_k)_{k\in\N}$ again and $\eta\in\Gamma$ of $w(\eta)=m_{2j}$ such
that $|\overrightarrow{e_{\eta}}^*(\sum_{k=1}^{n_{2j}}z_k)|\geq\frac{n_{2j}}{4m_{2j}}$.

For a suitable $\theta\in\R$ with $|\theta|\leq 2$ we can have that $\overrightarrow{e_{\eta}}^*(z)=1$ where
$z=\theta\sum_{k=1}^{n_{2j}}m_{2j}n_{2j}^{-1}z_k$. Using estimates that result from the basic inequality (Proposition \ref{basiceq}) it is easy to check that the pair $(z,\eta)$ is the desired
$(32C,2j,1)$ in $Y$. Since this is true for every $C>1$ the result follows.
\ep
We will be focused into finding finite sequences of $(C,j_k,\e)$ exact pairs $(z_k,\eta_k)_{k=1}^{n_{2j_0-1}}$ that have additional properties.
This type of sequences are called Dependent sequences (see \cite{AH}).

\begin{definition}
A finite sequence $(z_k)_{k=1}^{n_{2j_0-1}}$ in $\mathcal{Z}$ is called $(C,2j_0-1,\e)$ dependent sequence if there exist $p_1<q_1<p_2<q_2<\ldots<p_{n_{2j_0-1}}<q_{n_{2j_0-1}}$ if there exist $(\eta_k)_{k=1}^{n_{2j_0-1}}$ together with $(\xi_k)_{k=1}^{n_{2j_0-1}}$ such that $\eta_k\in\Gamma_{q_k}\setminus\Gamma_{p_k}$, $\xi_k\in\Delta_{q_k}$ the following are satisfied:
\begin{enumerate}
\item  with $\ran z_k\subset (p_k,q_k-1]$.
\item $(z_1,\eta_1)$ is $(C,4j_1-2,\e)$ exact pair and for each $k>1$ $(z_k,\eta_k)$ is $(C,4j_k,\e)$ exact pair.
\item The element $\gamma=\eta_{n_{2j_0-1}}$ has weight $m_{2j_0-1}$ and analysis \[\{p_i,q_i,\xi_i,\overrightarrow{e_{\eta_i}}^*\}_{i=1}^{n_{2j_0-1}}.\]
\end{enumerate}
Notice that by the definition of the sets $\Delta_n$ and the $(C,j,\e)$ exact pairs we have that $ w(\eta_1)=m_{4j_1-2}>n_{2j_0-1}^2$ and $w(\eta_{i+1})=m_{4j_{i+1}}$ where $j_{i+1}=\sigma(\xi_i)$ for $1\leq i\leq n_{2j_0-1}$.
\end{definition}
In a similar manner as in Lemma 6.4 of \cite{AH} every $(C,2j_0-1,\e)$ dependent sequence is a C-RIS. Applying basic inequality (Proposition \ref{basiceqcor}) we can have estimations on averages of dependent sequences which are helpful in order to prove the HI property in block subspaces of the space $\mathcal{Z}$.

\begin{proposition}\label{depend}
Let $(z_k)_{k=1}^{n_{2j_0-1}}$ be a $(C,2j_0-1,\e)$ dependent sequence. We set $z=\frac{1}{n_{2j_0-1}}\sum_{k=1}^{n_{2j_0-1}}z_k$ and $\tilde{z}=\frac{1}{n_{2j_0-1}}\sum_{k=1}^{n_{2j_0-1}}(-1)^kz_k$. Finally for $J$ subinterval of $[1,n_{2j_0-1}]$ we set $\tilde{z}_J=\sum_{k\in J}(-1)^kz_k$.
\begin{enumerate}
\item If $\e=1$, then $\|z\|\geq\frac{1}{m_{2j_0-1}}$ and $\|\tilde{z}\|\leq \frac{40C}{m_{2j_0-1}^2}$.
\item If $\e=0$ $\|z\|\leq\frac{30C}{m_{2j_0-1}^2}$.
\end{enumerate}
\end{proposition}
\bp The proof uses the same arguments as in \cite{AH} Prop 6.6, Lemma 8.9, therefore we will present only the basic steps for (i). For the first let $p_k,q_k,\eta_k,\xi_k$ that follow from the definition of a dependent sequence and let also $\gamma$ of $w(\gamma)=m_{2j_0-1}$ with analysis $\{p_k,q_k,\xi_k,\overrightarrow{e_{\eta_k}}^*\}_{k=1}^{n_{2j_0-1}}$. We note that since $\ran z_k\subset (p_k,q_k-1]$ and $\xi_k\in\Delta_{q_k}$ we have that $d_{\xi_k}^*(z_l)=0$ for every $k,l$ and thus $\overrightarrow{e_{\gamma}}^*(z)=\overrightarrow{e_{\eta_k}}^*(\frac{n_{2j_0-1}^{-1}}{m_{2j_0-1}}\sum_kz_k)=\frac{1}{m_{2j_0-1}}$.

For the second part of $(i)$ we estimate $\overrightarrow{e_{\gamma}}^*(\tilde{z})$ for every $\gamma$ with $w(\gamma)=m_{2j_0-1}$. Using a corresponding tree like property of the odd weight elements of $\Gamma$ as in \cite{AH}(Lemma 4.6) we deduce that for every $J$ subinterval of $[1,n_{2j_0-1}]$ and every $\gamma$ of weight $m_{2j_0-1}$ $|\overrightarrow{e_{\gamma}}^*(\tilde{z}_J)|\leq 4C$. As we mentioned above the dependent sequence $(z_k)_k$ is C-RIS and additionally it satisfies the moreover part of Corollary \ref{basiceqcor} (replacing C by 4C). We deduce that $\|\tilde{z}\|\leq \frac{10\cdot4C}{m_{2j_0-1}^2}=\frac{40C}{m_{2j_0-1}^2}$.
\ep

The next result uses the same arguments as Lemma 8.10 in \cite{AH}. We include a small proof for sake of completeness.
\begin{corollary}\label{HI}
Let $(x_n)_{n\in\N}$ be a block sequence in $\mathcal{Z}$. Then, the subspace $Z=\overline{<x_n:\ n\in\N>}$ of $\mathcal{Z}$ is HI (i.e. for every $Y_1$, $Y_2$ closed infinite
dimensional subspaces of $\mathcal{Z}$ $\dist(S_{Y_1},S_{Y_2})=0$, where $S_{Y_i}$ denotes the unit sphere of $Y_i$, $i=1,2$).
\end{corollary}
\bp
Assume that $Z=Y_1\oplus Y_2$, fix $\e>0$, $j_0\in\N$ such that $m_{2j_0-1}\e>2600$ and choose also $j_1\in\N$ such that
$m_{4j_1-2}>n_{2j_0-1}^2$. Without loss of generality we may assume that both $Y_1$, $Y_2$ are block subspaces.
Lemma \ref{exactness} implies that there exists $(65,m_{4j_1-2},1)$ exact pair $(z_1,\eta_1)$ in $Y_1$.
Let $q_1\in\N$ such that $\eta_1\in\Delta_{q_1}$ and for $p_1>\max\{q_1,\max\ran z_1\}$ we define $\xi_1\in\Delta_{p_1}$ as $\xi_1=(p_1,2j_0-1,\overrightarrow{e_{\eta_1}}^*)$.
Let $j_2=\sigma(\xi_1)$ and by Lemma \ref{exactness} we choose $(z_2,\eta_2)$ a $(65,4j_2,1)$ exact pair in $Y_2$ such that $\min\ran z_2>p_1$.
Let $q_2>p_1$ such that $\eta_2\in\Delta_{q_2}$ and for $p_2>\max\{q_2,\max\ran x_2\}$ we define $\xi_2\in\Delta_{p_2}$ as
$\xi_2=(p_2,\xi_1,2j_0-1,\overrightarrow{e_{\eta_2}}^*)$.
Inductively, we construct a $(65,2j_0-1,1)$ dependent sequence $(z_k)_{k=1}^{n_{2j_0-1}}$ such that $z_k\in Y_1$ for $k$ odd while $z_k\in Y_2$ for $k$ even.
Setting $z_1=\sum_{k\ odd}z_k\in Y_1$ and $z_2=\sum_{k even}z_k\in Y_2$ by Proposition \ref{depend} and the choice of $j_0$ we observe that $\|z_1-z_2\|<\e\|z_1+z_2\|$.
\ep

\section{Bounded linear operators on $Z$}
In this section we will study bounded and linear operators on $\mathcal{Z}=(\sum_nX_n)_{AH}$ for a fixed sequence of separable Banach spaces $(X_n,\|\cdot\|_n)_{n\in\N}$. We use an adaptation of basic techniques of \cite{AH} (Section 7) used into proving that their space $\mathfrak{X}_k$ has the "scalar plus compact" property( i.e. for every linear bounded operator $T$ on $\mathfrak{X}_k$ there exists a scalar $\lambda$ such that $T-\lambda I$ is compact).

A "weaker" type of compact operators is presented in the next definition.
\begin{definition}
A bounded and linear operator $K$ on $\mathcal{Z}$ is called horizontally
compact if for every bounded block sequence $(z_k)_{k\in\N}$ in
$Z$, with respect to $(Z_n)_{n\in\N}$, $\|K(z_k)\|\to 0$, or
equivalently, for every $\e>0$, there exists $k_{\e}\in\N$, such
that $\|K|_{\mathcal{Z}_{(k_{\e},\infty)}}\|<\e$, where $\mathcal{Z}_{(k_{\e},\infty)}=\sum_{n=k_{\e}+1}^{\infty}Z_n=P_{(k_{\e},\infty)}[\mathcal{Z}]$.
\end{definition}

In order to use some useful approximation arguments in \cite{AH} we need further notations. For a set $A$ we denote by $span_{\Q}A$ the set of all finite rational linear combinations of elements of $A$. It is known that every separable Banach space admits a bounded M-basis such that the set spanned by its biorthogonals is w* dense and in particular 1-norming for its dual space. For each $n\in\N$ we shall denote by $(x_{n,i})_{i\in\N}$ the M-basis of $X_n$ and by $\{(x_{n,i}^*)_{i\in\N}\}$ the set of its biorthogonals. We shall assume without loss of generality that $F_n=B_{X_n}^*\cap span_{\Q}\{x_{n,i}^*:\ i\in\N\}$ and we set $D_n=span_{\Q}\{x_{n,i}:\ \ i\in\N\}$.
Finally, for each $n\in\N$ we denote by $\ell^{\Q}_{\infty}(\Delta_n)$ the set $span_{\Q}\{e_n:\ n\in\N\}$ where $e_n$ is the usual unit vector.

 In the sequel for sake of simplicity we choose to work with a dense subset of the space $\mathcal{Z}$ rather than the whole space.

\begin{lemma}
For every $z\in\mathcal{Z}$ with $\ran z=(n,l]$ and every $\e>0$ there exists $w\in\mathcal{Z}$ such that $\ran z=\ran w$, $\|z-w\|<\e$ and if $\overrightarrow{v}=\supp_{loc}(w)$ then $\overrightarrow{v}\in \sum_{k=n+1}^l\oplus( D_k\oplus\ell^{\Q}_{\infty}(\Delta_k))$.
\end{lemma}
\bp
Let $\overrightarrow{u}\in\sum_{k=n+1}^l\oplus(X_k\oplus\ell_{\infty}(\Delta_k))$ such that $\supp_{loc}(z)=(\overrightarrow{0}_{1,n},\overrightarrow{u})$. We split $\overrightarrow{u}$ as $\overrightarrow{u}=\overrightarrow{u}'+\overrightarrow{u}''$ such that $\overrightarrow{u}'\in(\sum_{k=n+1}^l\oplus X_k)_{\infty}$ and $\overrightarrow{u}'' \in(\sum_{k=n+1}^l\oplus\ell_{\infty}(\Delta_k))_{\infty}$.  For each $k$ we find $v_k\in span_{\Q}\{x_{k,i}:\ i\in\N\}$, $y_k\in\ell_{\infty}(\Delta_k)$ with rational coordinates such that $\|\overrightarrow{u}'(k)-v_k\|_k\leq \frac{\e}{2}$ and $\|\overrightarrow{u}''(k)-y_k\|_k\leq \frac{\e}{2}$. Let $w=i_n(\overrightarrow{0}_{1,n},\overrightarrow{v})$ where $\overrightarrow{v}(k)=(v_k,y_k)$. We observe that $\|z-w\|\leq 2\|\overrightarrow{u}-\overrightarrow{v}\|_{\infty}<\e$ and combining all the above the proof is complete.
\ep

We set $Y'_n=\{x=i_n(\overrightarrow{u}): \ \ \overrightarrow{u}\in\sum_{k=1}^n\oplus( D_k\oplus\ell^{\Q}_{\infty}(\Delta_k)) \}$. A direct consequence of the above is that the union $\cup_n Y'_n$ is dense in $\mathcal{Z}$.

For the sequel we need the following notation.
\begin{notation}\label{splitf}
For $n<l$ we split $\overrightarrow{f}\in(\sum_{k=n}^l\oplus(X_k^*\oplus\ell_1(\Delta_k)))_1$ as $\overrightarrow{f}=\overrightarrow{f}'+\overrightarrow{f}''$ such that $\overrightarrow{f}'\in(\sum_{k=n}^l\oplus X_k^*)_1$ and $\overrightarrow{f}''\in(\sum_{k=n}^l\oplus\ell_1(\Delta_k))_1$. We follow the same method that we used for the local support of elements of $\mathcal{Z}$ (Remark \ref{split}), i.e. if $\overrightarrow{f}(k)=(x_k^*,b_k^*)\in (X_k^*\oplus\ell_1(\Delta_k))_1$ we set $\overrightarrow{f}'=(x_n^*,\ldots,x_l^*)$ and $\overrightarrow{f}''=(b_n^*,\ldots,b_l^*)$.
\end{notation}

\begin{remark}\label{approx}
Let $\overrightarrow{f}\in(\sum_{k=n}^l\oplus(X_k^*\oplus\ell_1(\Delta_k)))_1$ such that $\overrightarrow{f}'(k)\in F_k$ for every $n\leq k\leq l$, where $\overrightarrow{f}'$ as above. Then there exists $\gamma\in\Gamma$ such that $\overrightarrow{e_{\gamma}}^*(z)=\overrightarrow{f}'(z)$ for every $z\in Y'_l$.
Indeed, for every $n\leq k\leq l$ let $m_k\geq l$ such that $f'(k)\in F_k^{m_k}$ and let $m=\max\{m_k:k=n,\ldots,l\}$. Note that $m>l$, $\overrightarrow{f}'\in K_{m ,n}$ and let $\gamma=(m+1,n,\overrightarrow{f}',0)\in\Delta_{m+1}^0$. Since $m+1>l$ for $z\in Y'_l$ we have that $d_{\gamma}^*(z)=0$ and thus $e_{\gamma}^*(z)=c_{\gamma}^*(z)=\overrightarrow{f}'(z)$.
\end{remark}

\begin{lemma}\label{dualfunctionals}
Let $n<l$ and $z,w\in \mathcal{Z}$ such that $\ran z,\ran w\in (n,l]$ and $\dist(w,\R z)>\delta$. If we assume that $z\in Y'_l$ then there exists $q\geq l$, $\overrightarrow{b}^*\in B_{q,n}$ such that
that $\overrightarrow{b}^*(z)= 0$
$\overrightarrow{b}^*(w)>\frac{\delta}{4}$.
\end{lemma}
\bp
Let $\overrightarrow{u},\overrightarrow{v}\in\sum_{k=n+1}^l\oplus(X_k\oplus\ell_{\infty}(\Delta_k))$ such that $\supp_{loc}(z)=(\overrightarrow{0}_{1,n},\overrightarrow{u})$ and $\supp_{loc}(w)=(\overrightarrow{0}_{1,n},\overrightarrow{v})$. Notice that \[\|\overrightarrow{v}-\lambda\overrightarrow{u}\|_{\infty}=\|(\overrightarrow{0}_{1,n},\overrightarrow{v})-
\lambda(\overrightarrow{0}_{1,n},\overrightarrow{u})\|_{\infty}\geq\frac{1}{2}\|z-\lambda w\|>\frac{\delta}{2}.\]
Hence $\dist(\overrightarrow{v},\R\overrightarrow{u})\geq\frac{\delta}{2}$. By Hahn Banach Theorem there exists $\overrightarrow{f}\in(\sum_{k=n+1}^l\oplus(X_k^*\oplus\ell_1(\Delta_k)))_1$ such that $\overrightarrow{f}(\overrightarrow{u})=0$ and $\overrightarrow{f}(\overrightarrow{v})\geq\frac{\delta}{2}$. Considering Notation \ref{splitf} we split $\overrightarrow{f}=\overrightarrow{f}'+\overrightarrow{f}''$ where $\overrightarrow{f}'\in(\sum_{k=n+1}^l\oplus X_k^*)_1$ and $\overrightarrow{f}''\in(\sum_{k=n+1}^l\oplus\ell_1(\Delta_k))_1$. Since $z\in Y'_l$ we may assume that $\overrightarrow{f}'(k)\in F_k$ and $\overrightarrow{f}''(k)\in\ell_1^{\Q}(\Delta_k)$ for every $n+1\leq k\leq l$. By Remark \ref{approx} there exists $m\geq l$ and $\gamma\in\Delta_{m+1}^0$ such that $\overrightarrow{e_{\gamma}}^*(x)=\overrightarrow{f}'(x)$ for every $x\in\mathcal{Z}$. We set $\overrightarrow{b}^*=(\frac{1}{2}\overrightarrow{f}'',\overrightarrow{0}_{n+1,m},\frac{1}{2}\overrightarrow{e_{\gamma}}^*)$. Then $\overrightarrow{b}^*\in (\sum_{k=n+1}^{m+1}\oplus\ell_1(\Delta_k))_1$, $\|\overrightarrow{b}^*\|_1\leq 1$ and notice that each coordinate $\overrightarrow{b}^*(k)$ is a rational linear combination of $\{e_{\gamma}^*:\gamma\in\Delta_k\}$. Remark \ref{norming}(2) implies that there exists $q\geq m+1$ such that $\overrightarrow{b}^*\in B_{q,n}$. Observe that $\overrightarrow{b}^*(z)=\frac{1}{2}\overrightarrow{f}''(z)+\frac{1}{2}\overrightarrow{e_{\gamma}}^*(z)=\frac{1}{2}\overrightarrow{f}(z)=0$ and in a similar manner $\overrightarrow{b}^*(w)=\frac{1}{2}\overrightarrow{f}(w)\geq\frac{\delta}{4}$.
\ep

For the sequel of this section $\mathcal{Z}=(\sum\oplus X_n)_{BD}$ such that $(X_n,\|\cdot\|_n)_{n\in\N}$ is a sequence of separable Banach spaces with the additional property that either $\ell_1$ does not embed in
$X_n^*$ for every $n\in\N$ or $X_n$ admits the Schur property for every $n\in\N$.

Adapting the basic steps of Lemma 7.2 in \cite{AH} we arrive at the following result.
 \begin{lemma}\label{exact}
Let T be a bounded and linear operator on $\mathcal{Z}$ and $(w_k)_{k\in\N}$ be a $C$-RIS in $\cup_n Y'_n$ such that $\dist(Tw_k,\R w_k)>\delta>0$ for every
$n\in\N$.
Then, for all $j,p\in\N$, there exist $z\in [w_k:\ \ k\in\N]$ and $\eta\in\Delta_q$, $q>p$ such that
$(z,\eta)$ is $(16C,2j,0)$ exact pair, $\|I-P_{(p,q]}Tz\|\leq\delta m_{2j}^{-1}$ and $P_{(p,q]}^*e_{\eta}^*(Tz)>\frac{\delta}{8}$.
\end{lemma}
\bp Let $j,p\in\N$. Repeatedly applying Proposition \ref{blocking} we may assume passing to a subsequence that there exists $p<r_1<l_1<\ldots<r_k<l_k<r_{k+1}<\ldots$ such that $\ran w_k\subset(r_k,l_k]$ and $\|(I-P_{(r_k,l_k]})Tw_k\|\leq\frac{\delta}{80m_{2j}}$ for every $k$. It follows that $\dist(P_{(r_k,l_k]}Tw_k,\R
w_k)>\frac{7\delta}{16}$. By Lemma \ref{dualfunctionals} we can find $q_k\geq l_k$ and $\overrightarrow{b}_k^*\in B_{q_k,r_k}$
that $\overrightarrow{b}^*_k(w_k)=0$
$\overrightarrow{b}^*_k(P_{(r_k,l_k]}Tw_k)>\frac{7\delta}{64}$.

Passing to a subsequence if necessary we may additionally require that
$r_k<q_k+1<r_{k+1}<\ldots$ and let $z=\frac{m_{2j}}{n_{2j}}\sum_{k=1}^{n_{2j}}w_k$. Assuming that $2j<r_1$, we can
recursively choose $\xi_k\in\Delta_{q_k+1}$ with $w(\xi_k)=m_{2j}$
and construct an element $\eta\in\Gamma$ with analysis
$\{r_k,q_k,\xi_k,\overrightarrow{b}^*_k,\}_{k=1}^{n_{2j}}$. Similarly as in \cite{AH}(Lemma 7.2) it is proved that the pair $(z,\eta)$ satisfies
the hypothesis. \ep

Repeatedly and carefully applications of the above lemma imply the following result that is an adaptation of Proposition 7.3 in \cite{AH}.
For sake of completeness we give a slight description of the
proof.
\begin{lemma}\label{RIS}
Let $T:\mathcal{Z}\to \mathcal{Z}$ be a linear bounded operator. Then, for every RIS $(w_k)_{k\in\N}$ in $\mathcal{Z}$, $\dist(Tw_k,\R w_k)\to 0$.
\end{lemma}
\begin{proof}
It is enough to prove it for every RIS in $\cup_nY'_n$. Suppose on the contrary that there exists an RIS $(w_k)_{k\in\N}$ in $\cup_nY'_n$ and $\delta>0$ such that $\dist(Tw_k,\R w_k)>\delta$ for every $k\in\N$. Let also $j_0$ that will be determined later and choose $j_1$ such that $m_{4j_1-2}>n_{2j_0-1}^2$. Applying Lemma \ref{exact} for $j=2j_1-1$ and $p_1=1$ we can find $q_1>1$ and a $(16C,4j_1-2,0)$ exact pair $(z_1,\eta_1)$ such that $\eta_1 \in\Delta_{q_1}$, $\|I-P_{(1,q_1]}Tz_1\|\leq\delta m_{4j_1-2}^{-1}$ and $(\overrightarrow{e_{\eta_1}}^*(P_{(p_1,q_1]}Tz_1)>\frac{\delta}{8}$. Let $\xi_1=(q_1+1,m_{2j_0-1},\overrightarrow{e_{\eta_1}}^*)$, $j_2=\sigma(\xi_1)$ and apply again Lemma \ref{exact} for $j=4j_2$ and $p_2=q_1+1$. Inductively we construct
$p_1<q_1<p_2<q_2<\ldots$, a sequence $(z_i,\eta_i)_{i=1}^{l}$ such that each $(z_i,\eta_i)$ is $(16C,4j_i,0)$ exact pair with $\ran z_i\subset (p_i,q_i]$,
$\eta_i \in\Delta_{q_i}$, $\|I-P_{(p_i,q_i]}Tz_i\|\leq\delta m_{4j_i}^{-1}$ and $\overrightarrow{e_{\eta}}^*_i(P_{(p_i,q_i]}Tz_i)>\frac{\delta}{8}$. Observe that $(z_i)_{i=1}^{n_{2j_0-1}}$ is a $(16C,2j_0-1,0)$ dependent sequence and let
$\gamma\in \Gamma$ with analysis
$\{p_i,q_i,\xi_i,\overrightarrow{e_{\eta_i}}^*\}_{i=1}^{n_{2j_0-1}}$. We notice that $d_{\xi_i}^*(z_j)=0$ for every $i,j$ while for every $i$ \[\overrightarrow{e_{\gamma}}^*(Tz_i)\geq \frac{1}{m_{2j_0-1}}(\overrightarrow{e_{\eta}}^*_i(P_{(p_i,q_i]}Tz_i)-\|I-P_{(p_i,q_i]}Tz_i\|)> \frac{\delta}{8m_{2j_0-1}}-\frac{\delta}{n_{2j_0-1}}.\]
Setting $z=\frac{1}{n_{2j_0-1}}\sum_{i=1}^{n_{2j_0-1}}z_i$ we have the
estimation \[\overrightarrow{e_{\gamma}}^*(Tz)=n_{2j_0-1}^{-1}\sum_{i=1}^{n_{2j_0-1}}\overrightarrow{e_{\gamma}}^*(Tz_i)>\frac{\delta}{16m_{2j_0-1}}.\] Proposition \ref{depend} yields that $\|Tz\|\leq \frac{30\cdot 16C\|T\|}{m_{2j_0-1}^2}$. Now $j_0$
 can be suitable chosen in order to conclude that $\overrightarrow{e_{\gamma}}^*(Tz)>\|Tz\|$ yielding a contradiction.
\end{proof}

All the above yield
\begin{proposition}\label{operators}
Let $T$ be a linear and bounded operator on $\mathcal{Z}$. Then, there
exists a scalar $\lambda$ such that the operator $T-\lambda I$ is horizontally
compact.
\end{proposition}

\bp Let $(w_k)_{k\in\N}$ be a normalized RIS in $\mathcal{Z}$. Lemma \ref{RIS} yields that there exist scalars $\lambda_k$ such that $\|Tw_k-\lambda_kw_k\|\to 0$. An easy argument as
in \cite{AH}, Theorem 7.4 implies that the scalars $\lambda_k$
tend to some scalar $\lambda$ which does not depend on the choice of $(w_k)_{k\in\N}$. By Proposition \label{RISblock} we deduce that
$\|(T-\lambda I)z_k\|\to 0$, for every
bounded block sequence $(z_k)_{n\in\N}$ in $\mathcal{Z}$.
\ep

\section{Quasi Prime AH-$\mathcal{L}^{\infty}$ sums of Banach spaces}
We now study AH sums for specific sequence of Banach spaces. In particular, we denote by $\mathcal{Z}_p$ for $1\leq p<\infty$ the AH-$\mathcal{L}_{\infty}$ sum of
the sequence
$(X_n,\|\cdot\|_n)_{n\in\N}$ such that $X_n=\ell_p$ for every $n\in\N$. A direct consequence of the proceeding study is the following:

\begin{corollary}\label{scplcom}
\begin{enumerate}
\item[(i)] $\mathcal{Z}_p$ is non isomorphic to $\ell_p$.
\item[(ii)] For every bounded and linear operator $T$ on $\mathcal{Z}_p$ there
exists a scalar $\lambda$ such that the operator $T-\lambda I$ is horizontally
compact.
\end{enumerate}
\end{corollary}
\bp  For $(i)$ we observe that by Proposition \ref{HI} the space $\mathcal{Z}_p$ contains an HI subspace and thus cannot be isomorphic to $\ell_p$. The second $(ii)$ is a direct application of Proposition \ref{operators}.
\ep

As it is shown in the next proposition $\ell_p$ is isomorphic to complemented subspaces of $\mathcal{Z}_p$.

\begin{proposition}
For every $k_0\in\N$ the image $P_{[1,k_0]}(\mathcal{Z}_p)$ is isomorphic to $\ell_p$.
\end{proposition}
\bp We set $Y_{k_0}=P_{[1,k_0]}(\mathcal{Z}_p)$ and we recall that $Y_{k_0}$ is isomorphic to $U_{k_0}=(\sum_{k=1}^{k_0}\oplus(\ell_p\oplus\ell_{\infty}(\Delta_k)))_{\infty}$. It is easy to see that $U_{k_0}$ is isomorphic to $\ell_p$. More precisely there exists a constant $C_{k_0}\geq 1$ such that $\|\overrightarrow{u}\|_{\infty}\leq \|\overrightarrow{u}\|_p\leq C_{k_0} \|\overrightarrow{u}\|_{\infty}$ for every $\overrightarrow{u}\in U_{k_0}$. We mention that $C_k\to\infty$.
\ep

\begin{remark}\label{elp}
In a similar manner as above we have that $(\sum_{k=m}^{k}\oplus(\ell_p\oplus\ell_{\infty}(\Delta_k)))_{\infty}$ is isomorphic to $\ell_p$ for every $m\leq k$. Hence, Lemma \ref{dual1} we deduce that $P_{[m,k]}(\mathcal{Z}_p)$ is isomorphic to $\ell_p$ for every $m\leq k$.
\end{remark}

In order to show that the spaces $\mathcal{Z}_p$ are strictly quasi prime we need the following lemmas.

\begin{lemma}\label{hisubspace}
Let $1\leq p<\infty$ and suppose that $\mathcal{Z}_p\simeq U\oplus V$. Then, there exists $k_0\in\N$ such that either $P_{[1,k_0]}|_U$ or $P_{[1,k_0]}|_V$
is an isomorphic embedding.
\end{lemma}
\bp
Let $P:\mathcal{Z}_p\to\mathcal{Z}_p$ be a projection onto $U$. By Corollary \ref{scplcom} we have that there exists a scalar $\lambda$ such that $P=\lambda I+K$,
where $K$ is a horizontally compact operator on $\mathcal{Z}_p$.
If $\lambda=0$ we have that $U=K[\mathcal{Z}_p]$ and by the definition of the horizontally compact operator the result trivially holds for $U$. Otherwise,
$\lambda\neq 0$ and we claim that in this case $P_{[1,k_0]}|_V$ is an isomorphic embedding. Indeed, if we assume the opposite we can find a normalized
sequence $(v_n)_{n\in\N}$ in $V$ and a block sequence $(x_n)_{n\in\N}$ in $\mathcal{Z}_p$ such that $\|v_n-x_n\|\to 0$. Since $\|K(x_n)\|\to 0$ we have that $\|K(v_n)\|\to 0$.
A simple observation is that $|\lambda|-\|K(v_n)\|\leq\|P(v_n)\|$ which yields that $\lambda=0$ contradiction our initial assumption.
\ep
The arguments of the next lemma are adapted from \cite{AH} (Lemma 3).
\begin{lemma}\label{isomorph}
Let $1\leq p<\infty$ and $Y$ be a subspace of $\mathcal{Z}_p$ for which there exist $k_0\in\N$ such that $P_{[1,k_0]}|_Y$ is an isomorphism.
Then, the following hold:
\begin{enumerate}
\item If $p=1$ and $Y$ is complemented in $\mathcal{Z}_1$ by a
projection $P$, then for every $\e>0$, there exists $k_{\e}\in\N$
such that $\|P(z)\|<\e\|x\|$ for every $z\in
P_{(k_{\e},\infty)}[\mathcal{Z}_1]$.
\item If $p>1$, then for every $\e>0$ there
exists $k_{\e}\in\N$ such that
$\|P_{(k_{\e},\infty)}(y)\|<\e\|y\|$ for every $y\in Y$.
\end{enumerate}
\end{lemma}
\bp In Case (1) if we assume the opposite, then there exists
$\e>0$ such that for every $k\in\N$, there exists $x_k\in
\mathcal{Z}_{1(k,\infty)}$ such that $\|z_k\|=1$ and
$\|P(z_k)\|\geq \e$. Using a sliding hump argument we may also
assume that the sequence $(z_k)_{k\in\N}$ is block.
This implies that both $(z_k)_{k\in\N}$ and $(P(z_k))_{k\in\N}$
are weakly null. Since $Y$ is isomorphically
embedded into $Z_{1[1,k_0]}$ which is isomorphic to $\ell_1$ and
by the Schur property of $\ell_1$ we deduce that $\|P(z_k)\|\to
0$, which is a contradiction.

In Case (2), contradicting the
assumption again, we have that there exists $\e>0$ such that for
every $k\in\N$, a normalized sequence $(y_k)_{k\in\N}$ in $Y$ and
a sequence of successive intervals $(I_k)_{k\in\N}$ such that
$\|P_{I_k}(y_k)\|\geq \e$. Now, since $Y$ is
isomorphically embedded into $\mathcal{Z}_{p[1,k_0]}\simeq\ell_p$, passing to a subsequence if necessary, we
may assume that the sequence $z_k=y_{2k}-y_{2k-1}$ is w-null and
passing again to a subsequence we have that $(z_k)_k$ is
equivalent to the unit standard vector basis of $\ell_p$. Thus,
there exists a constant $C>0$ such that
$\|\frac{1}{n}\sum_{k=1}^nz_k\|\leq C
\frac{n^{\frac{1}{p}}}{n}$ for every $n\in\N$.

Fix $j\in\N$ such that
$\frac{\e}{16m_{2j}}>C\frac{n_{2j}^{\frac{1}{p}}}{n_{2j}}$. Passing to a subsequence we can have that $\|P_{(\min I_{k+1},\infty)}y_k\|\to 0$ and thus we may assume that $\|P_{I_{2k}}z_k\|\geq \frac{\e}{2}$ for every $k\in\N$. Considering Lemma \ref{wnull} for the chosen $j$, $J_k=I_{2k}$ and $\delta=\e/2$ there exist elements $(z_{k_i})_{i=1}^{n_{2j}}$ and $\gamma\in\Gamma$ of weight $w(\gamma)=m_{2j}$ such that $|e_{\gamma}^*(\sum_{i=1}^{n_{2j}}z_{k_i})|\geq \frac{\e n_{2j}}{16m_{2j}}$. It follows that $\|\frac{1}{n_{2j}}\sum_{i=1}^{n_{2j}}z_{k_i}\|\geq \frac{\e}{16m_{2j}}>C\frac{n_{2j}^{\frac{1}{p}}}{n_{2j}}$ yielding a contradiction.

\ep

\begin{proposition}
The spaces $\mathcal{Z}_p$, $1\leq p<\infty$ are strictly quasi prime.
\end{proposition}
\bp Let $1\leq p<\infty$ and set $\mathcal{Z}=\mathcal{Z}_p$. Suppose that $\mathcal{Z}\simeq V\oplus U$ and let
$P:\mathcal{Z}\to\mathcal{Z}$ such that $Im P=V$. By Lemma \ref{hisubspace} we may assume that $V$ does not contain an HI subspace and hence
Lemma \ref{isomorph} implies that there exists $k_0\in\N$ such that $\|PP_{(k_0,\infty)}|_V\|\leq 1$ and
$\|P_{(k_0,\infty)}P|_{\mathcal{Z}_{(k_0,\infty)}}\|\leq 1$. This yields that the operators $P:P_{[1,k_0]}(V)\to V$ and $P_{[1,k_0]}:V\to P_{[1,k_0]}(V)$
are inventible as well as $S=P_{[1,k_0]}P:P_{[1,k_0]}(V)\to P_{[1,k_0]}(V)$. Let $Q:\mathcal{Z}_{[1,k_0]}\to V$ defined as $Q=S^{-1}\circ P_{[1,k_0]}P|_{\mathcal{Z}_{[1,k_0]}}$. Then $Q$ is a projection onto $P_{[1,k_0]}(V)\simeq V$
and since $\mathcal{Z}_{[1,k_0]}\simeq \ell_p$ we have that $V\simeq\ell_p$. The following factorization
\[\mathcal{Z}_{(k_0,\infty)}\stackrel{I-P}{\to} U\stackrel{P_{(k_0,\infty)}}{\to}\mathcal{Z}_{(k_0,\infty)}\] yields that
$\mathcal{Z}_{(k_0,\infty)}$ is isomorphic with a complemented subspace of $U$. In particular, by Remark \ref{elp} we deduce that $\ell_p$ is isomorphic to a
complemented subspace of $U$. Therefore, $U\simeq\ell_p\oplus Z\simeq \ell_p\oplus\ell_p\oplus Z\simeq \ell_p\oplus U\simeq V\oplus U\simeq \mathcal{Z}_p$.
\ep

\section{Complemented subspaces of $\mathcal{Z}_p^n$}
In this section we shall study the complemented subspaces of a
finite powers $\mathcal{Z}_p^n=\sum_{i=1}^n\oplus
\mathcal{Z}_p(i)$, endowed with the
supremum norm as an external one. It is clear that since
$\mathcal{Z}_p$ is strictly quasi prime $\mathcal{Z}_p^n\simeq
\ell_p\oplus \mathcal{Z}_p^n$. Therefore, we are interested for
the non trivial complemented subspaces of $\mathcal{Z}_p^n$ that
are not isomorphic to $\ell_p$.
\begin{notation}
For the sequel, for $I\subset\N$ and $L\subset\{1,2,\ldots,n\}$ we denote by
$P_I^L:\mathcal{Z}_p^n\to \mathcal{Z}_p^n$, the natural
projections defined as
$P_I^L(\sum_{i=1}^nz_i)=\sum_{i\in L}P_I(z_i)$, for
$\sum_{i=1}^nz_i\in \mathcal{Z}_p^n$. In the case that $L=\{1,2,\ldots,n\}$ we simply write $P_I^n$. For technical reasons, for a subspace $Y$ of $\mathcal{Z}_p^n$ we will write $Y_I$ instead of $P_I^n(Y)$. Moreover, we say that a sequence $(x_k)_{k\in\N}$ in $\mathcal{Z}_p^n$ is block if $\max_{i=1,\ldots,n}\{\ran x_k(i)\}<\min_{i=1,\ldots,n}\{\ran x_{k+1}(i)\}$, where $x_k(i)\in \mathcal{Z}_p(i)$.
\end{notation}

The following lemma is a generalization of Lemma \ref{isomorph}.

\begin{lemma}\label{isomorphic}
Let $1\leq p<\infty$ and $Y$ be a subspace of $\mathcal{Z}_p^n$ for which there exist $k_0\in\N$ such that $P_{[1,k_0]}^n|_Y$ is an isomorphism. Then, the following hold:
\begin{enumerate}
\item If $p=1$ and $Y$ is complemented in $\mathcal{Z}_1^n$ by a
projection $P$, then for every $\e>0$, there exists $k_{\e}\in\N$
such that $\|P(x)\|<\e\|x\|$ for every $x\in
P_{(k_{\e},\infty)}^n[\mathcal{Z}_1^n]$. \item If $p>1$, then for every $\e>0$ there
exists $k_{\e}\in\N$ such that
$\|P_{(k_{\e},\infty)}^n(y)\|<\e\|y\|$ for every $y\in Y$.
\end{enumerate}
\end{lemma}
\bp Case (1) is proved using the same arguments as in Case (1) of Lemma \ref{isomorph}.
In Case (2), contradicting the
assumption again, we have that there exists $\e>0$ such that for
every $k\in\N$, a normalized sequence $(y_k)_{k\in\N}$ in $Y$ and
a sequence of successive intervals $(I_k)_{k\in\N}$ such that
$\|P_{I_k}^n(y_k)\|\geq \e$. For every $k$, let $n_k\in L$
satisfying $\|P_{I_k}^n(y_k)(n_k)\|\geq \e$. Then, there exists
$M\in[\N]$ such that $n_k=n_0$ for every $k\in M$. Using arguments as in the proof of Lemma \ref{isomorph}(2) we arrive to a contradiction.
\ep

\begin{notation}\label{projection}
Let $T:\mathcal{Z}_p^n\to \mathcal{Z}_p^n$ be a linear and bounded
operator. Then, $T$ is written into the form $T=(T_{i,j})_{1\leq
i,j\leq n}$, where $T_{i,j}:Z_{p(j)}\to \mathcal{Z}_{p(i)}$.
Proposition \ref{scplcom}(1) yields that for every $1\leq i,j\leq n$
there exists a scalar $\lambda_{i,j}$ such that
$T_{i,j}=\lambda_{i,j}I_{i,j}+K_{i,j}$, where
$I_{i,j}:\mathcal{Z}_{p(j)}\to \mathcal{Z}_{p(i)}$ is the natural
identity map and $K_{i,j}:\mathcal{Z}_{p(j)}\to
\mathcal{Z}_{p(i)}$ is a horizontally compact operator. Setting
$\Lambda=(\lambda_{i,j})_{1\leq i,j\leq n}$, $\Lambda
I=(\lambda_{i,j}I_{i,j})_{1\leq i,j\leq n}$ and
$K=(K_{i,j})_{1\leq i,j\leq n}$ we have that $T =\Lambda I+K$.
\end{notation}

\begin{lemma}\label{projectiononR}
Let $m\leq n$ and $T:\mathcal{Z}_p^n\to \mathcal{Z}_p^m$ be a
linear and bounded operator of the form $T =\Lambda I+K$ as above. Then,
\begin{enumerate}
\item If $m=n$ and $T$ is a projection, then $\Lambda=(\lambda_{i,j})_{1\leq i,j\leq n}$ is a projection on $\R^n$.
\item If $n>m$, then $T$ cannot be an isomorphic embedding.
\end{enumerate}
\end{lemma}
\bp For the first assume, on the contrary, that $\Lambda^2\neq
\Lambda$. Then, there exists $\tilde{0}\neq
\tilde{a}=(a_1,\ldots,a_n)\in\mathbb{R}^n$ such that
$\Lambda^2(\tilde{a})\neq\Lambda(\tilde{a})$. We may assume that
$|a_i|\leq 1$ for all $1\leq i\leq n$ and that there exists $i_0$
such that $|a_{i_0}|=1$. We consider the supremum norm on $\mathbb{R}^n$ and let $0<\epsilon=\|\Lambda^2(\tilde{a})-\Lambda(\tilde{a})\|$. Note that $P^2=\Lambda^2I+K'$, where $K'=(\Lambda I)
K+K(\Lambda I)+K^2$. Since $K_{i,j}$ are horizontally compact we can find $k_0\in\N$ and
$x\in \mathcal{Z}_{p(k_0,\infty)}$ such that $\|x\|=1$ and
both
$\|K(\tilde{x})\|<\frac{\epsilon}{4}$,
$\|K'(\tilde{x})\|\leq\frac{\epsilon}{4}$. Let $\tilde{x_i}\in
\mathcal{Z}^n$ defined as $\tilde{x_i}(j)=\begin{cases} 0,\ \text{if}\
i\neq j\\ x,\ \text{if}\ i=j \end{cases}$ and we set
$\tilde{x}=\sum_{i=1}^na_i\tilde{x_i}$. Clearly,
$\tilde{x}$ belongs in $\mathcal{Z}_{p(k_0,\infty)}^n$, $\|\tilde{x}\|=1$ and
\[\|P^2(\tilde{x})-P(\tilde{x})\|\geq
\|\Lambda^2I(\tilde{x})-\Lambda
I(\tilde{x})\|-\|K(\tilde{x})-K'(\tilde{x})\|\geq\epsilon-\frac{\epsilon}{2}=\frac{\epsilon}{2},\]
which contradicts the fact that $P^2=P$.

For the second part we refer the reader to \cite{AR}, Prop.3 \ep

\bpr\label{complement_Z^n} Let W be an infinite dimensional complemented subspace of
$\mathcal{Z}_p^n$. Then, either $W\simeq\ell_p$ or there exists a non-empty set $L\subset\{1,\ldots, n\}$
such that W is isomorphic to $\mathcal{Z}_p^L(=\sum_{i\in L}\oplus \mathcal{Z}_p)$.
\epr
 \bp Let
$P:\mathcal{Z}_p^n\to \mathcal{Z}_p^n$ be a projection onto $W$,
i.e $P[\mathcal{Z}_p^n]=W$. Lemma \ref{projectiononR} (1), yields that
$\Lambda=(\lambda_{i,j})_{i,j}$ is a projection on $\R^n$. Thus, there exists an inventible matrix $A:\R^n\to\R^n$ of
the form $A=(a_{i,j})_{1\leq i,j\leq n}$ such that $A\Lambda
A^{-1}=(\tilde{\lambda}_{i,j})_{1\leq i,j\leq n}$ where
$\tilde{\lambda}_{i,j}=\begin{cases} 0\ \text{or}\ 1,\ \text{if}\ i=j\\ 1,\ \ \ \ \ \ \text{if}\
i\neq j.
\end{cases}$
Considering the inventible operator $\tilde{A}=(a_{i,j}I_{i,j})_{1\leq i,j\leq n}$ we set
$\tilde{P}=\tilde{A}P\tilde{A}^{-1}$. It is easy to see that
$\tilde{P}:\mathcal{Z}_p^n\to \mathcal{Z}_p^n$ is a projection of the form $\tilde{P}=(\tilde{\lambda}_{i,j}I_{i,j})_{i,j}+\tilde{K}$ where $\tilde{K}=\tilde{A}K\tilde{A}^{-1}=(\tilde{K}_{i,j})_{1\leq i,j\leq n}$ such
that $\tilde{K}_{i,j}$ is horizontally compact for every $1\leq
i,j\leq n$ and
$W\simeq\tilde{P}[\mathcal{Z}_p^n]$. Therefore we may assume that $P=\tilde{P}$ (i.e.$\tilde{K}=K$, $\tilde{\Lambda}=\Lambda$ ).

We set $L=\{i:\ \mu_{i,i}\neq 0\}$.We distinguish into the following cases:\\
Case 1: $L=\emptyset$.
This implies that $P=K$. We claim that there exists $k_0\in\N$ such that $P_{[1,k_0]}^n|_W$ is an isomorphism. Indeed if not, then we can find a normalized
sequence $(w_k)_{k\in\N}$ and a block sequence $(x_k)_{k\in\N}$ in $\mathcal{Z}_p^n$ such that $\|x_k-w_k\|\to 0$.
Since $\|K(x_k)\|\to 0$ we have that $\|w_k\|=\|P(w_k)\|=\|K(w_k)\|\to 0$, a contradiction.
Lemma \ref{isomorphic} implies that there exist $\ell_0\in\N$ such that $\|PP_{(l_0,\infty)}^n|_W\|<\frac{1}{2}$. We conclude that $W\simeq W_{[1,l_0]}^n$ and setting $T:P_{[1,l_0]}^n\circ
P|_{W_{[1,l_0]}^n}:W_{[1,l_0]}^n\to W_{[1,l_0]}^n$ we have that
$T$ is an isomorphism and $T^{-1}\circ P_{[1,l_0]}^n\circ
P|_{P_{[1,l_0]}^n\mathcal{Z}_p^n}:\mathcal{Z}_{p[1,l_0]}^n\to
W_{[1,l_0]}^n$ is a projection on $W_{[1,l_0]}^n$. Since
$\mathcal{Z}_{p[1,l_0]}^n\simeq\ell_p$, the result follows.

Case 2: $L\neq\emptyset$. In this case we prove that
$W\simeq\mathcal{Z}_p^L$. In particular, we
shall prove that $W\simeq\mathcal{Z}_p^L\oplus
Y$, where $Y\simeq\ell_p$ and since $\mathcal{Z}_p$ is quasi
prime, i.e $\mathcal{Z}_p\simeq \mathcal{Z}_p\oplus\ell_p$, the
result will follow. We recall that $K=(K_{i,j})_{i,j}$ where each $K_{i,j}$ is horizontally compact and hence we can find $k_0\in\N$ such that $\|K|_{\mathcal{Z}_{p(k_0,\infty)}^n}\|<\frac{1}{4}$. It follows that the operator
\[P_{(k_0,\infty)}^L\circ P|_{\mathcal{Z}_{p(k_0,\infty)}^L}:\mathcal{Z}_{p(k_0,\infty)}^L\to\mathcal{Z}_{p(k_0,\infty)}^L\] is inventible.
Moreover the factorization of the above operator yields that $\mathcal{Z}_p^L\simeq\mathcal{Z}_{p(k_0,\infty)}^L$ is isomorphic to a complemented subspace of $W$ and thus $W\simeq\mathcal{Z}_p^L\oplus
Y$ as promised. Then, it is easy to see that there exists $\ell_0\in\N$ such that
 $P_{[1,l_0]}^n|_Y$ is an isomorphic embedding.
Working similarly as in Case 1 and using Lemma \ref{isomorphic} we obtain that $Y$ is complemented in $\ell_p$ and so isomorphic to $\ell_p$.
\ep

Lemma \ref{projectiononR}(2) and Propositions \ref{complement_Z^n} yield the following.
\begin{corollary}
The spaces $\mathcal{Z}_p^n$, for $1\leq p<\infty$ admit exactly $n+1$, up to
isomorphism, complemented subspace.
\end{corollary}

\end{document}